\documentclass[12pt]{amsart}

\usepackage{fullpage}

\allowdisplaybreaks

\newcommand{\bG}{{\mathbf{G}}}
\newcommand{\bB}{{\mathbf{B}}}
\newcommand{\bT}{{\mathbf{T}}}
\newcommand{\bU}{{\mathbf{U}}}
\newcommand{\bH}{{\mathbf{H}}}

\newcommand{\Z}{{\mathbb{Z}}}
\newcommand{\C}{{\mathbb{C}}}
\newcommand{\F}{{\mathbb{F}}}

\newcommand{\T}{{\mathbb{T}}}

\newcommand{\lin}{{\mathrm{lin}}}
\newcommand{\GL}{{\operatorname{GL}}}
\newcommand{\Tr}{\operatorname{Tr}}
\newcommand{\height}{\operatorname{ht}}

\newcommand{\dfour}[4]{\begin{footnotesize}\ensuremath{\hspace{-0.1cm}\begin{array}{ccc}#1\vspace{-2mm}\\ &\hspace{-2mm} #3 &\hspace{-2mm} #4 \vspace{-2mm} \\ #2 \\\end{array}\hspace{-0.1cm}}\end{footnotesize}}

\newcommand{\bs}{\mathbf s}
\newcommand{\bt}{\mathbf t}

\newcommand{\cC}{{\mathcal{C}}}
\newcommand{\cF}{{\mathcal{F}}}

\newtheorem{theorem}{Theorem}[section]
\newtheorem{lemma}[theorem]{Lemma}
\newtheorem{proposition}[theorem]{Proposition}

\theoremstyle{remark}

\newtheorem{remark}[theorem]{Remark}

\numberwithin{equation}{section}

\usepackage[pdftex,colorlinks,hypertexnames=false,linkcolor=blue,urlcolor=blue,citecolor=blue]{hyperref}

\begin{document}

\title[The generic character table of $UD_4(p^a)$]{The generic character table of a Sylow $p$-subgroup of a finite Chevalley group of type $D_4$}

\author{Simon M.~Goodwin, Tung Le and Kay Magaard}

\address{School of Mathematics, University of Birmingham,
Birmingham, B15 2TT, U.K.} \email{s.m.goodwin@bham.ac.uk} \email{k.magaard@bham.ac.uk}

\address{Department of Mathematics and Applied Mathematics, University of Pretoria, Pretoria 0002, South Africa}
\email{lttung96@yahoo.com}

\begin{abstract}
Let $U$ be a Sylow $p$-subgroup of the finite Chevalley group of type $D_4$
over the field of $q$ elements, where $q$ is a power of a prime $p$. We describe a construction of
the generic character table of $U$.
\end{abstract}

\maketitle

\section{Introduction} \label{sec:intro}

Let $p$ be a prime and let $q$ be a power of $p$.  Let $G$ be
a finite reductive group over the field of $q$ elements,
and let $U$ be a Sylow $p$-subgroup of $G$.
The aim of this paper is to describe a construction of the
{\em generic character table} of $U$ for the case where $G$ is split and of type $D_4$.
This generic character table can be determined immediately
from the sequence of propositions and lemmas in Section \ref{sec:main}.
The construction is ``uniform'' over primes $p > 2$, but we observe
differences for the bad prime $p = 2$.

There has been much research into constructing generic character tables of $G$.  For $G$
of small rank these have been programmed into the computing system CHEVIE, \cite{CHEVIE}.
More recently there has been interest in generic character tables for Borel and parabolic
subgroups of $G$, see for example \cite{Hi1}, \cite{Hi3}, \cite{HH1}, \cite{HH1b} and \cite{HN2},
where many cases for $G$ of $\F_q$-rank $3$ or less are considered.
These papers and similar methods have led to applications in determining decomposition
numbers in the modular character theory  of $G$,
see for example, \cite{AH2}, \cite{Hi2}, \cite{Hi4}, \cite{HH2}, \cite{HN1}, \cite{OW} and \cite{Wa}.

There has been a great deal of research interest in the character theory of $U$,
with motivation coming from understanding how this character theory varies with $q$.

Interest in properties of $U$ for $G = \GL_n(q)$ goes back to
G.~Higman \cite{Higman}, where the conjugacy classes of $U$ are considered.
In \cite{Lehrer}, G.~Lehrer determined how discrete series characters
decompose when restricted to $U$ for $G = \GL_n(q)$.
This led to a study of the irreducible characters of $U$ and in
particular it was conjectured that the degrees of the irreducible characters are always
a power of $q$, and that the number of characters of degree $q^d$, for $d \in \Z_{\ge 0}$ is
given by a polynomial in $q$ with integer coefficients;
this refines a well known conjecture attributed to Higman.
There has been further work in this direction, for $G = \GL_n(q)$, by I.~M.~Isaacs,
and more generally on the character theory of algebra groups, see \cite{IsaacsCharAlg}
and \cite{Is}.  In particular, Isaacs verified that the degree of any irreducible character
of an algebra group over $\F_q$ is a power of $q$.
For other recent developments we refer, for example, to the work of A.~Evseev \cite{EV}, E.~Marberg \cite{Marb},
and I.~Pak and A.~Soffer \cite{PS}; in particular, we remark that the results of
\cite{PS} suggest that the aforementioned conjecture on the number of characters being
given by a polynomial in $q$ is false.
An interesting aspect of the character theory of $U$ is the theory of supercharacters, which was first
studied by Andr\'e, see for example \cite{An}; this theory was fully developed by Diaconis and Isaacs in \cite{DI}.

There has been much interest on the complex irreducible
characters of Sylow $p$-subgroups of finite reductive groups of other types.
For instance, it has been proved that for $G$ of type $B$, $C$ or $D$ the degrees of all irreducible
characters of $U(q)$ are powers of $q$ if and only if $p \ne 2$, see \cite{Pr2}, \cite{Sa2} and \cite{Szeg}.
There has been much further interest in determining certain families of irreducible character,
see for example \cite{GMP}, \cite{Marj}, \cite{Sa3} and \cite{Szec}.
Recently, a notion of supercharacters of $U$ for $G$ of types $B$, $C$ and $D$ was introduced by Andr\'e and
Neto in \cite{AN}.
For $p$ larger than the Coxeter number of $G$, the Kirillov orbit method
gives a bijection between the coadjoint orbits of $U$ on its Lie algebra and
the irreducible characters of $U$; in \cite{GMR2} an algorithm is given, which is used
to calculate a parametrization of these coadjoint orbits when $G$ is split and of rank less than or equal to 8, excluding type $E_8$.
We refer the reader also to
\cite{LM} for other results on the parametrization
of irreducible characters of $U$.

From now on we suppose $G$ is a finite Chevalley group of type $D_4$.
The two key
ingredients that enable us to calculate the generic
table of $U$ are the papers
\cite{HLM} and \cite{GoodwinRoehrle}, which provide
a parametrization of the irreducible characters
and the conjugacy classes of $U$; we also use \cite{BRGO} for the conjugacy classes
for the case $p=2$.
Our construction of this generic character table
is given in such a way that it is straightforward
to convert it to a computer file that can be used in future
calculations.

Our parametrization of the characters and conjugacy
classes of $U$ partitions them into families that are independent of $q$, as can
be observed in Tables \ref{tab:representatives} and \ref{tab:irrU}.
We have refined the parameterizations from \cite{GoodwinRoehrle} and \cite{HLM},
so that they are stable under the automorphisms
of $G$ coming from the graph automorphisms of the Dynkin diagram;
this is likely to be of significant benefit for applications.
Our calculations are based
on developing the methods from \cite{HLM}.  A key additional ingredient is
that we modify those
constructions making them more explicit by  realizing the
irreducible characters of $U$ as induced characters
of linear characters of certain subgroups.

By the {\em generic character table} we mean the
table with rows labelled by the families of the characters, and columns
labelled by the families of representatives of conjugacy classes.
The values in the table are given by expressions depending
on the parameters for the families and on $q$.
This can be viewed as an extension
of the ideas discussed above regarding how the character theory of
$U$ varies with $q$.

We observe that the characters and their values are uniform for $p > 2$,
whereas for $p=2$ there are differences. More explicitly the parametrization
of the family of irreducible characters $\cF_{8,9,10}$ is different
for $p=2$, and although the parametrization of the characters in the families $\cF_{11}$
and $\cF_{12}$ is not different for $p=2$, the values of these
characters do differ.  These differences
give some explanation of why $p=2$ is a bad prime for $D_4$.

We summarize our results in the following theorem.

\begin{theorem} \label{mainthm}
Let $G$ be a a finite Chevalley group of type $D_4$ over $\F_q$ and let $U$ be
a Sylow $p$-subgroup of $G$.  The generic character table of $U$ is determined
from: the parametrization of the conjugacy classes given in Section \ref{sec:classes}; the
parametrization of the irreducible characters given in Section \ref{sec:chars}; and
the character values given in Section \ref{sec:main}.
\end{theorem}

The methods in this paper serve as a model, which can be applied
to determine generic character tables of Sylow $p$-subgroups of finite groups of
Lie type, where $p$ is the defining characteristic.

We outline the structure of paper.  First we recall the required theory and notation
regarding the finite Chevalley group of type $D_4$ in Section \ref{sec:nota}.  Then
in Sections \ref{sec:classes} and \ref{sec:chars}, we recall the results from \cite{BRGO}, \cite{GMR} and
\cite{HLM} giving the parametrization of the conjugacy classes and irreducible characters of
$U$.  The majority of the work is done in Section \ref{sec:main}, where we determine
the character values.

\medskip
\noindent
\textbf{Acknowledgements:}
We would like to thank the referee for many helpful suggestions that led to
improvements in the paper.  We also thank A.~Paolini for pointing out
some corrections.  Part of this research was
completed during a visit of the second author to the University of Birmingham; we thank
the LMS for a grant to cover his expenses.

\section{Notation and Setup}
\label{sec:nota}

Let $p$ be a prime, let $q$ be a power of $p$, and denote the finite
field with $q$ elements by $\F_q$.
We let $\bG$ denote a simply connected simple algebraic group of type $D_4$ defined and split
over $\F_p$; in other words $\bG$ is $\mathrm{Spin}_8$ defined over $\F_p$.
  Let $\bB$ be a Borel subgroup of $\bG$ defined over $\F_p$, let $\bU$
be the unipotent radical of $\bB$, and $\bT \subseteq \bB$ be a maximal
torus defined over $\F_p$.  For a subgroup $\bH$ of $\bG$ defined over $\F_p$, we write $H = \bH(q)$
for the group of $\F_q$-rational points of $\bH$.  So $G = \bG(q)$
is a finite Chevalley group of type $D_4$ and $U = \bU(q)$ is a Sylow $p$-subgroup
of $G$.
We cite \cite[Sec.\ 3]{DM}, as a reference
for the theory of algebraic groups over finite fields, and for the
terminology used here.

We note that $G = \mathrm{Spin}^+_8(q)$ is not a simple group, but
that $G/Z(G) \cong P\Omega_8^+(q)$ is a simple group.
As we are mainly interested in $U$ in this paper, which is isomorphic
to its image in $G/Z(G)$, this distinction causes us no harm.

Let $\Phi$ be a root system of $\bG$ with respect to $\bT$ and let $\Phi^+$ be
the system of positive roots determined by $B$.
We write $\{\alpha_1,
\alpha_2, \alpha_3, \alpha_4\}$ for the corresponding base of simple roots ordered
so that the Dynkin diagram of $\Phi$ is
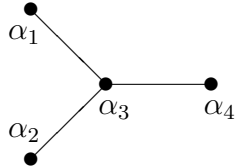
\begin{center}
\setlength{\unitlength}{1cm}
\begin{picture}(3.4,3)
\thinlines
\put(0.5,0.5){\circle*{0.17}}
\put(1.5,1.5){\circle*{0.17}}
\put(2.9,1.5){\circle*{0.17}}
\put(0.5,2.5){\circle*{0.17}}
\put(0.5,0.5){\line( 1, 1){1}}
\put(1.5,1.5){\line( 1, 0){1.4}}
\put(0.5,2.5){\line( 1, -1){1}}
\put(0.2,0.8){$\alpha_2$}
\put(1.4,1.1){$\alpha_3$}
\put(2.8,1.1){$\alpha_4$}
\put(0.2,2.1){$\alpha_1$}
\end{picture}
\end{center}
For $\alpha = \sum_{i=1}^4 a_i \alpha_i \in \Phi$ we write $\height(\alpha)
= \sum_{i=1}^4 a_i$ for the height of $\alpha$.
We use the notation
\dfour1121 for the
root $\alpha_1 + \alpha_2 + 2 \alpha_3 + \alpha_4 \in \Phi^+$ and we use a
similar notation for the remaining positive roots.
We enumerate the roots in $\Phi^+$ as shown in Table~\ref{tab:posroots} below.

\begin{table}[!ht]
\begin{center}
\begin{tabular}{c|llll}
\hline
Height & Roots &&& \\
\hline
1 & $\alpha_1$  & $\alpha_2$  & $\alpha_3$   & $\alpha_4$ \\
\hline
2 & $\alpha_5 := $ \dfour1010 & $\alpha_6 := $ \dfour0110 & $\alpha_7 := $ \dfour0011 & \\
\hline
3 & $\alpha_8 := $ \dfour1110 & $\alpha_9 := $ \dfour1011 & $\alpha_{10} := $ \dfour0111 & \\
\hline
4 & $\alpha_{11} := $ \dfour1111 &&& \\
\hline
5 & $\alpha_{12} := $ \dfour1121 &&& \\
\hline
\end{tabular}
\end{center}
\medskip
\caption{Positive roots of the root system $\Phi$ of type $D_4$.}
\label{tab:posroots}

\end{table}

The triality automorphism of the Dynkin diagram can be lifted to
an automorphism of $\bG$ that stabilizes $\bT$ and $\bB$,
and restricts to an automorphism of $G$.  We refer the reader
to \cite[Section~4.2]{Carter1} for more information and let $\tau$
denote this automorphism of $G$.  In a similar
way we can define three automorphisms of $G$ corresponding to the
involutory automorphisms of the Dynkin diagram given by
switching two of the outer vertices.  We denote these by
$\sigma_{12}$, $\sigma_{14}$ and $\sigma_{24}$ in the obvious
way.

The group $G$ is generated by the root elements $x_\alpha(t)$ for
$\alpha \in \Phi$ and $t \in \F_q$.  We may choose
these root elements so that, for $\alpha, \beta \in \Phi$ such that
$\alpha + \beta \in \Phi$, we have $[x_\alpha(t), x_\beta(u)] = x_{\alpha+\beta}(\pm tu)$
and the signs are invariant under $\tau$ and the $\sigma_{ij}$.
For positive roots, we use the
abbreviation $x_i(t) := x_{\alpha_i}(t)$ and let $X_i = \{x_i(t) \mid t \in \F_q\}$, for $i=1,2,\dots,12$.
 We have $\tau(x_i(t)) = x_{i\rho}(t)$ where $\rho$ is the permutation
$\rho = (1,4,2)(5,7,6)(8,9,10)$;
and have a similar description for the automorphisms $\sigma_{ij}$.
The commutators $[x_i(t), x_j(u)] = x_i(t)^{-1} x_j(u)^{-1} x_i(t)
x_j(u)$
are given in Table~\ref{tab:commrelD4}. All $[x_i(t),
x_j(u)]$ not listed in this table are equal to~$1$.  We refer the reader
to \cite[Chapter 4]{Carter1} for details on root elements and
commutator relations.
\begin{table}[!ht]
\begin{tabular}{llllll}
$\left[x_1(t), x_3(u)\right]$ & $=$ & $x_5(tu)$, \, &
$\left[x_1(t), x_6(u)\right]$ & $=$ & $x_8(tu)$, \\
$\left[x_1(t), x_7(u)\right]$ & $=$ & $x_9(tu)$, \, &
$\left[x_1(t), x_{10}(u)\right]$ & $=$ & $x_{11}(tu)$, \\
$\left[x_2(t), x_3(u)\right]$ & $=$ & $x_6(tu)$, \, &
$\left[x_2(t), x_5(u)\right]$ & $=$ & $x_8(tu)$,\\
$\left[x_2(t), x_7(u)\right]$ & $=$ & $x_{10}(tu)$, \, &
$\left[x_2(t), x_9(u)\right]$ & $=$ & $x_{11}(tu)$,\\
$\left[x_3(t), x_4(u)\right]$ & $=$ & $x_7(-tu)$, \, &
$\left[x_3(t), x_{11}(u)\right]$ & $=$ & $x_{12}(tu)$,\\
$\left[x_4(t), x_5(u)\right]$ & $=$ & $x_9(tu)$, \, &
$\left[x_4(t), x_6(u)\right]$ & $=$ & $x_{10}(tu)$,\\
$\left[x_4(t), x_8(u)\right]$ & $=$ & $x_{11}(tu)$, \, &
$\left[x_5(t), x_{10}(u)\right]$ & $=$ & $x_{12}(-tu)$,\\
$\left[x_6(t), x_9(u)\right]$ & $=$ & $x_{12}(-tu)$,\, &
$\left[x_7(t), x_8(u)\right]$ & $=$ & $x_{12}(-tu)$.
\end{tabular}
\medskip
\caption{Commutator relations for type $D_4$.} \label{tab:commrelD4}
\end{table}

We have that $U$ is the subgroup of $G$ generated by the
elements $x_i(t_i)$ for $i=1,2, \dots, 12$ and $t_i \in \F_q$.
In fact we have that $U = \prod_{i=1}^{12} X_i$.
It turns out that it is convenient for us to write our elements
of $U$ with an element from $X_3$ at the front, and we use the notation
\begin{equation} \label{e:x(t)}
x(\bt) = x_3(t_3) x_1(t_1) x_2(t_2) x_4(t_4) x_5(t_5) \cdots x_{12}(t_{12}),
\end{equation}
where $\bt = (t_1,\dots,t_{12})$.

There is a further piece of notation that we use frequently in the sequel,
we define the normal subgroups $M_i$ of $U$ for $i=1,\dots,13$ by
\begin{align} \label{e:Mi}
M_1 & = U,  \quad  \\
M_2 = X_1X_2\prod_{j=4}^{12} X_j, \nonumber \\
M_3 &= X_2\prod_{j=4}^{12} X_j \text{ and } \nonumber \\
M_i &= \prod_{j=i}^{12} X_j \quad \text{for $i=4,\dots,13$}.
\end{align}

\section{Conjugacy Classes and their Representatives}
\label{sec:classes}

In \cite{GoodwinRoehrle}, an algorithm for calculating the conjugacy
classes of Sylow $p$-subgroups of finite Chevalley groups in good
characteristic is described; so we assume $p > 2$ for the first part of this section.
This algorithm was implemented in GAP4
\cite{GAP4} and, in particular, used to determine the conjugacy
classes of $U = U(q)$. The conjugacy class representatives are given
as {\em minimal representatives} as defined in \cite{Goodwin} and
discussed below.

(For ease of exposition we temporarily let
$x_1 = x_{\alpha_3}$, $x_2 = x_{\alpha_1}$ and $x_3 = x_{\alpha_2}$
in the next paragraph.)

The idea behind the calculation of minimal representatives is to
recursively determine the orbits of $U$ on $U/M_{i+1}$ for $i =
1,\dots,12$; recall that $M_{i+1}$ is defined in \eqref{e:Mi}.
To explain this we consider the set
$A(x,i) = \{xx_i(t)M_{i+1} \mid t \in \F_q\} \subseteq U/M_{i+1}$
for $x \in U$. The
centralizer $C_U(xM_i)$ acts on $A(x,i)$ by conjugation. The key
results for the theory of minimal representatives are \cite[Lem.\ 5.1
and Prop.\ 6.2]{Goodwin}, which imply the following dichotomy:
\begin{itemize}
\item[(I)]  all elements of $A(x,i)$ are conjugate under $C_U(xM_i)$; or
\item[(R)] no two elements of $A(x,i)$ are conjugate under $C_U(xM_i)$.
\end{itemize}
We say that $i$ is an {\em inert point} of $x$ if (I) holds and $i$
is a {\em ramification point} of $x$ if (R) holds.  We say $x(\bt)$
as defined in \eqref{e:x(t)} is the {\em minimal representative} of
its conjugacy class if $t_i = 0$ whenever $i$ is an inert point of
$x(\bt)$.  Thanks to \cite[Prop.\ 5.4 and Prop.\ 6.2]{Goodwin}, the minimal
representatives of conjugacy classes give a complete set of representatives of the
conjugacy classes in $U$.

{\tiny

\begin{table}[!ht]
\begin{tabular}{l|l|l|l|l}
\hline \hline
Label & Representative & Conditions & Number & Centralizer \\
\hline
$\cC_{1,2,3,4}$ & $x_3(a_3)x_1(a_1)x_2(a_2)x_4(a_4)$ & & $(q-1)^4$ & $q^4$\\
\hline \hline
$\cC_{1,2,3}$ & $x_3(a_3)x_1(a_1)x_2(a_2)x_9(b_9)x_{10}(b_{10})$ & $b_9/a_1 + b_{10}/a_2 = 0$ & $(q-1)^3q$ & $q^5$\\
\hline
$\cC_{1,3,4}$ & $x_3(a_3)x_1(a_1)x_4(a_4)x_8(b_8)x_{10}(b_{10})$ & $b_8/a_1 + b_{10}/a_4 = 0$ & $(q-1)^3q$ & $q^5$\\
\hline
$\cC_{2,3,4}$ & $x_3(a_3)x_2(a_1)x_4(a_4)x_8(b_8)x_9(b_9)$ & $b_8/a_2 + b_9/a_3 = 0$ & $(q-1)^3q$ & $q^5$\\
\hline \hline
$\cC_{1,3}$ & $x_3(a_3)x_1(a_1)x_{10}(b_{10})$ & & $(q-1)^2q$ & $q^5$ \\
\hline
$\cC_{2,3}$ & $x_3(a_3)x_2(a_2)x_9(b_9)$ & & $(q-1)^2q$ & $q^5$ \\
\hline
$\cC_{3,4}$ & $x_3(a_3)x_4(a_4)x_8(b_8)$ & & $(q-1)^2q$ & $q^5$ \\
\hline \hline
$\cC_{1,2,4,q^6}$ ($p > 3$)& $x_1(a_1)x_2(a_2)x_4(a_4)x_5(c_5)x_6(c_6)x_7(c_7)$  &  $c_5/a_1 + c_6/a_2 + c_7/a_4 = 0$ & $(q-1)^3(q^2-1)$ & $q^6$ \\
$\cC_{1,2,4,q^6}$ ($p = 3$)& $x_1(a_1)x_2(a_2)x_4(a_4)x_6(c_6)x_7(c_7)$  &  & $(q-1)^3(q^2-1)$ & $q^6$ \\
$\cC_{1,2,4,q^7}$ & $x_1(a_1)x_2(a_2)x_4(a_4)x_{12}(b_{12})$ & & $(q-1)^3q$ & $q^7$ \\
\hline \hline
$\cC_{1,2,q^6}$ & $x_1(a_1)x_2(a_2)x_5(c_5)x_6(c_6)x_7(c_7)$ & $c_5/a_1 + c_6/a_2 = 0$ & $(q-1)^2(q^2-1)$ & $q^6$ \\
$\cC_{1,2,q^7}$ & $x_1(a_1)x_2(a_2)x_9(a_9)x_{10}(a_{10})$ & $a_9/a_1 + a_{10}/a_2 = 0$ & $(q-1)^3$ & $q^7$ \\
$\cC_{1,2,q^8}$ & $x_1(a_1)x_2(a_2)x_{12}(b_{12})$ & & $(q-1)^2q$ & $q^8$ \\
\hline
$\cC_{1,4,q^6}$ & $x_1(a_1)x_4(a_4)x_5(c_5)x_6(c_6)x_7(c_7)$ & $c_5/a_1 + c_7/a_4 = 0$ & $(q-1)^2(q^2-1)$ & $q^6$ \\
$\cC_{1,4,q^7}$ & $x_1(a_1)x_4(a_4)x_8(a_8)x_{10}(a_{10})$ & $a_8/a_1 + a_{10}/a_4 = 0$ & $(q-1)^3$ & $q^7$ \\
$\cC_{1,4,q^8}$ & $x_1(a_1)x_4(a_4)x_{12}(b_{12})$ & & $(q-1)^2q$ & $q^8$ \\
\hline
$\cC_{2,4,q^6}$ & $x_2(a_2)x_4(a_4)x_5(c_5)x_6(c_6)x_7(c_7)$ & $c_6/a_2 + c_7/a_4 = 0$ & $(q-1)^2(q^2-1)$ & $q^6$ \\
$\cC_{2,4,q^7}$ & $x_2(a_2)x_4(a_4)x_8(a_8)x_9(a_9)$ & $a_8/a_2 + a_9/a_4 = 0$ & $(q-1)^3$ & $q^7$ \\
$\cC_{2,4,q^8}$ & $x_2(a_2)x_4(a_2)x_{12}(b_{12})$ & & $(q-1)^2q$ & $q^8$ \\
\hline \hline
$\cC_{1,q^6}$ & $x_1(a_1)x_6(c_6)x_7(c_7)$ & & $(q-1)(q^2-1)$ & $q^6$ \\
$\cC_{1,q^7}$ & $x_1(a_1)x_{10}(a_{10})$ & & $(q-1)^2$ & $q^7$ \\
$\cC_{1,q^8}$ & $x_1(a_1)x_{12}(b_{12})$ & & $(q-1)q$ & $q^8$ \\
\hline
$\cC_{2,q^6}$ & $x_2(a_2)x_5(c_5)x_7(c_7)$ & & $(q-1)(q^2-1)$ & $q^6$ \\
$\cC_{2,q^7}$ & $x_2(a_2)x_9(a_9)$ & & $(q-1)^2$ & $q^7$ \\
$\cC_{2,q^8}$ & $x_2(a_2)x_{12}(b_{12})$ & & $(q-1)q$ & $q^8$ \\
\hline
$\cC_{4,q^6}$ & $x_4(a_4)x_5(c_5)x_6(c_6)$ & & $(q-1)(q^2-1)$ & $q^6$ \\
$\cC_{4,q^7}$ & $x_4(a_4)x_8(a_8)$ & & $(q-1)^2$ & $q^7$ \\
$\cC_{4,q^8}$ & $x_4(a_4)x_{12}(b_{12})$ & & $(q-1)q$ & $q^8$ \\
\hline \hline
$\cC_3$ & $x_3(a_3)x_8(b_8)x_9(b_9)x_{10}(b_{10})x_{11}(b_{11})$ & & $(q-1)q^4$ & $q^8$ \\
\hline \hline
$\cC_{5,6,7}$ & $x_5(a_5)x_6(a_6)x_7(a_7)x_{11}(b_{11})$ & & $(q-1)^3q$ & $q^8$ \\
\hline \hline
$\cC_{5,6,q^8}$ & $x_5(a_5)x_6(a_6)x_9(a_9)x_{10}(a_{10})$ & $a_9/a_5 + a_{10}/a_6 = 0$ & $(q-1)^3$ & $q^8$ \\
$\cC_{5,6,q^9}$ & $x_5(a_5)x_6(a_6)x_{11}(b_{11})$ & & $(q-1)^2q$ & $q^9$ \\
\hline
$\cC_{5,7,q^8}$ & $x_5(a_5)x_7(a_7)x_8(a_8)x_{10}(a_{10})$ & $a_8/a_5 + a_{10}/a_7 = 0$ & $(q-1)^3$ & $q^8$ \\
$\cC_{5,7,q^9}$ & $x_5(a_5)x_7(a_7)x_{11}(b_{11})$ & & $(q-1)^2q$ & $q^9$ \\
\hline
$\cC_{6,7,q^8}$ & $x_6(a_6)x_7(a_7)x_8(a_8)x_9(a_9)$ & $a_8/a_6 + a_9/a_7 = 0$ & $(q-1)^3$ & $q^8$ \\
$\cC_{6,7,q^9}$ & $x_6(a_6)x_7(a_7)x_{11}(b_{11})$ & & $(q-1)^2q$ & $q^9$ \\
\hline \hline
$\cC_{5,q^8}$ & $x_5(a_5)x_{10}(a_{10})$ & & $(q-1)^2$ & $q^8$ \\
$\cC_{5,q^9}$ & $x_5(a_5)x_{11}(b_{11})$ & & $(q-1)q$ & $q^9$ \\
\hline
$\cC_{6,q^8}$ & $x_6(a_6)x_9(a_9)$ & & $(q-1)^2$ & $q^8$ \\
$\cC_{6,q^9}$ & $x_6(a_6)x_{11}(b_{11})$ & & $(q-1)q$ & $q^9$ \\
\hline
$\cC_{7,q^8}$ & $x_7(a_7)x_8(a_8)$ & & $(q-1)^2$ & $q^8$ \\
$\cC_{7,q^9}$ & $x_7(a_7)x_{11}(b_{11})$ & & $(q-1)q$ & $q^9$ \\
 \hline \hline
 $\cC_{8,9,10}$ & $x_8(c_8)x_9(c_9)x_{10}(c_{10})$ & & $q^3-1$ & $q^{10}$ \\
\hline \hline
$\cC_{11}$ & $x_{11}(a_{11})$ & & $q-1$ & $q^{11}$ \\
\hline \hline
$\cC_{12}$ & $x_{12}(b_{12})$ & & $q$ & $q^{12}$ \\
\hline \hline
\end{tabular}
\medskip
\caption{Representatives of conjugacy classes in $U(q)$ for $p > 2$}
\label{tab:representatives}
\end{table}
}

In Table \ref{tab:representatives}, we present the minimal representatives
of the conjugacy classes in $U(q)$.
Starting with the minimal representatives calculated as in
\cite{GoodwinRoehrle}, we have parameterized the conjugacy
classes giving representatives partitioned into families.
This has been done to reduce the number of families required
to what appears to be an optimal number.

The first column of Table \ref{tab:representatives} gives a label for the
family, then in the second column a parametrization of the elements
of the family is given, where the notational convention is that an $a_i$ is an element of
$\F_q^\times = \F_q \setminus \{0\}$, a $b_i$ is any element of $\F_q$, and for
a sequence of the form $c_{i_1},\dots,c_{i_r}$ the $c_{i_j}$ are
elements of $\F_q$, which are not all zero.  These parameters have
to satisfy the conditions given in the third column.  In
the fourth column we give the number of elements in the family.  The
families are constructed so that the size of the centralizer of an
element in the family does not depend on the choice, and the final
column gives this centralizer size.

The families of
representatives have been chosen, so that they are permuted by the automorphisms $\tau$,
$\sigma_{12}$, $\sigma_{14}$ and $\sigma_{24}$.  Thus we can read off the
action of these automorphisms on the conjugacy classes.  However,
for the family, $\cC_{1,2,4,q^6}$, it is not possible for $p=3$ to get this stability
under the triality automorphism and we have to distinguish the cases
$p = 3$ and $p > 3$;
we note that the family given for $p = 3$
is actually also valid for $p > 3$.

To end this section we explain the differences in the parametrization of the conjugacy classes
for $p=2$, where the theory of minimal representatives fails.  First we note that there are a number of families, where there is a relation between
parameters involving a sum of two terms and this has to be replaced by just setting one of the
parameters equal to zero.  For example in $\cC_{1,2,3}$, we have to replace the representatives simply with
$x_3(a_3)x_1(a_1)x_2(a_2)x_{10}(b_{10})$ and there are no conditions.  This does
lead to the families of representatives not being stable under the automorphisms,
but this cannot be helped.

The more significant differences are as follows.
\begin{itemize}
\item The family $\cC_{1,2,3,4}$ is replaced by the family
$\cC_{1,2,3,4}^{p=2}$.  The representatives are of the form
$x_3(a_3)x_1(a_1)x_2(a_2)x_4(a_4)x_{10}(d_{10})$,
where $d_{10}$ is either 0, or an element of $\F_q$ not in the image of the
map $t \mapsto a_3a_2a_4(t^2+t)$.  So there are $2(q-1)^4$ representatives
in this family.  Their centralizers have order $2q^4$.
\item The family $\cC_{1,2,4,q^6}$ should be included as in the $p=3$ case in the table above.
\item The family $\cC_{1,2,4,q^7}$ is replaced by two families.
\begin{itemize}
  \item $\cC_{1,2,4,2q^7}^{p=2}$. The representatives are of the form
  $
  x_1(a_1)x_2(a_2)x_4(a_4)x_{10}(a_{10})x_{12}(d_{12}),
  $
  where $d_{12}$ is either 0, or an element of $\F_q$ not in the image of the
  map $t \mapsto a_1a_2a_4t^2+a_1a_{10}t$.  So there are $2(q-1)^4$ representatives
  in this family.  Their centralizers have order $2q^7$.
  \item $\cC_{1,2,4,q^7}^{p=2}$.  The representatives are of the form
  $
  x_1(a_1)x_2(a_2)x_4(a_4).
  $
  There are $(q-1)^3$ representatives
  in this family and their centralizers have order $q^7$.
\end{itemize}
\item The family $\cC_{5,6,7}$ is replaced by two new families namely.
\begin{itemize}
\item $\cC_{5,6,7,2q^8}^{p=2}$. The representatives are of the form
$x_5(a_5)x_6(a_6)x_7(a_7)x_{10}(a_{10})x_{11}(d_{11})$,
where $d_{11}$ is either 0, or an element of $\F_q$ not in the image of the
map $t \mapsto a_5a_6a_7t^2+a_5a_{10}t$.
There are $2(q-1)^4$ representatives
in this family and their centralizers have order $2q^8$.
\item $\cC_{5,6,7,q^8}^{p=2}$. The representatives are of the form
$x_5(a_5)x_6(a_6)x_7(a_7)$.  The number of representatives in this family is $(q-1)^3$
and their centralizers have order $q^8$.
\end{itemize}
\end{itemize}

\section{The irreducible characters of \texorpdfstring{$U$}{U}}
\label{sec:chars}

In \cite{HLM}, the irreducible characters of $U$ were determined giving a partition
into families, which are given in Table \ref{tab:irrU} for the case $p >2$.
We have modified these constructions to make them more explicit, so that
we can realize them all as induced characters from linear characters
of certain subgroups of $U$.  In the next section we
calculate the values of the characters,
and in doing so we explain their construction.

\begin{table}[!ht]

\begin{center}
\renewcommand{\arraystretch}{1.5}
\begin{tabular}{l|l|l|l|l}
\hline
\hline
Family & Notation & Conditions & Number & Degree
\\
\hline
\hline
$\mathcal{F}_{12}$ & $\chi_{12}^{a_{12},b_1,b_2,b_4}$ & & $q^3(q-1)$ & $q^4$
\\
\hline
$\mathcal{F}_{11}$ & $\chi_{11}^{a_{11},b_5,b_6,b_7,b_3}$  & & $q^4(q-1)$ & $q^3$
\\
\hline
\hline
$\mathcal{F}_{8,9,10}$  & $\chi_{8,9,10}^{a_8,a_9,a_{10},b_3}$  & & $q(q-1)^3$ & $q^3$
\\
\hline
\hline
$\mathcal{F}_{8,9}$ & $\chi_{8,9,q^3}^{a_8,a_9,a_6,a_7}$  & $a_6/a_8+a_7/a_9=0$ & $(q-1)^3$ & $q^3$\\
& $\chi_{8,9,q^2}^{a_8,a_9,b_2,b_3,b_4}$ & $b_2/a_8+b_4/a_9=0$ & $q^2(q-1)^2$ & $q^2$
\\
\hline
$\mathcal{F}_{8,10}$ & $\chi_{8,10,q^3}^{a_8,a_{10},a_5,a_7}$ & $a_5/a_8+a_7/a_{10}=0$ & $(q-1)^3$ & $q^3$\\
& $\chi_{8,10,q^2}^{a_8,a_{10},b_1,b_3,b_4}$ & $b_1/a_8+b_4/a_{10}=0$ & $q^2(q-1)^2$ & $q^2$
\\
\hline
$\mathcal{F}_{9,10}$ & $\chi_{9,10,q^3}^{a_9,a_{10},a_5,a_6}$ & $a_5/a_9+a_6/a_{10}=0$ & $(q-1)^3$ & $q^3$\\
& $\chi_{9,10,q^2}^{a_9,a_{10},b_1,b_2,b_3}$ & $b_1/a_9+b_2/a_{10}=0$ & $q^2(q-1)^2$ & $q^2$
\\
\hline \hline
$\mathcal{F}_{8}$ & $\chi_{8,q^3}^{a_8,a_7}$ & & $(q-1)^2$ & $q^3$\\
& $\chi_{8,q^2}^{a_8,b_3,b_4}$ & & $q^2(q-1)$ & $q^2$
\\
\hline
$\mathcal{F}_{9}$ & $\chi_{9,q^3}^{a_9,a_6}$ & & $(q-1)^2$ & $q^3$\\
& $\chi_{9,q^2}^{a_9,b_2,b_3}$ & & $q^2(q-1)$ & $q^2$
\\
\hline
$\mathcal{F}_{10}$ & $\chi_{10,q^3}^{a_{10},a_5}$ & & $(q-1)^2$ & $q^3$\\
& $\chi_{10,q^2}^{a_{10},b_1,b_2}$  & & $q^2(q-1)$ &  $q^2$
\\
\hline
\hline
$\mathcal{F}_{5,6,7}$ ($p>3$) & $\chi_{5,6,7}^{a_5,a_6,a_7,b_1,b_2,b_4}$ & $b_1/a_5+b_2/a_6+b_4/a_7=0$ & $q^2(q-1)^3$ & $q$ \\
$\mathcal{F}_{5,6,7}$ ($p=3$) & $\chi_{5,6,7}^{a_5,a_6,a_7,b_2,b_4}$ &  & $q^2(q-1)^3$ & $q$
\\
\hline \hline
$\mathcal{F}_{5,6}$ & $\chi_{5,6}^{a_5,a_6,b_1,b_2,b_4}$ & $b_1/a_5+b_2/a_6=0$ & $q^2(q-1)^2$ & $q$
\\
\hline
$\mathcal{F}_{5,7}$ & $\chi_{5,7}^{a_5,a_7,b_1,b_2,b_4}$ & $b_1/a_5+b_4/a_7=0$ & $q^2(q-1)^2$ & $q$
\\
\hline
$\mathcal{F}_{6,7}$ & $\chi_{6,7}^{a_6,a_7,b_1,b_2,b_4}$ & $b_2/a_6+b_4/a_7=0$ & $q^2(q-1)^2$ & $q$
\\
\hline \hline
$\mathcal{F}_{5}$ & $\chi_{5}^{a_5,b_2,b_4}$ & & $q^2(q-1)$ & $q$
\\
\hline
$\mathcal{F}_{6}$ & $\chi_{6}^{a_6,b_1,b_4}$ & & $q^2(q-1)$ & $q$
\\
\hline
$\mathcal{F}_{7}$ & $\chi_{7}^{a_7,b_1,b_2}$ & & $q^2(q-1)$ & $q$
\\
\hline \hline
$\mathcal{F}_{\lin}$ &
$\chi_{\lin}^{b_1,b_2,b_3,b_4}$ & & $q^4$ & $1$
\\
\hline \hline
\end{tabular}
\end{center}
\medskip
\caption{The irreducible characters of $U$ for $p>2$.} \label{tab:irrU}
\end{table}

We explain the contents of Table \ref{tab:irrU}.  The first column
gives the name of the family and the second column gives the
names of the characters in that family.  The notational convention
is that parameters denoted by $a_i$ range over $\F_q^\times = \F_q \setminus \{0\}$, and
parameters denoted by $b_i$ range over $\F_q$.   The subscripts in
the notation give an idea of ``where these parameters act'', as
can be understood from the descriptions of how the characters are constructed
given in the next section. For some families there is a relation between some of the parameters
and this is given in the third column.
The fourth column records
the number of characters in that family.   The families are constructed
in such a way that the degree of the characters in a family is constant,
and the fifth column gives this degree.

The triality automorphism $\tau$ of $G$, permutes the families of
characters by acting on the subscripts of the families
by the permutation $\rho$; similarly the automorphisms
$\sigma_{12}$, $\sigma_{14}$ and $\sigma_{24}$ act on the families.
As mentioned earlier, we have updated the parametrization from
\cite[\S4]{HLM}, so that we can read off the action
of these automorphisms on the characters.
Though it is not possible to do this for the family $\cF_{5,6,7}$ for $p = 3$.

We end this section by noting that for $p=2$, the parametrization of irreducible characters is slightly different.
As explained in \cite{HLM}, the family denoted $\cF_{8,9,10}$  is more
complicated and has to be replaced by
the family $\cF_{8,9,10}^{p=2}$, which consists of two types of characters as explained below and in more detail
at the end of \S\ref{ss:8910}.
\begin{itemize}
\item The characters $\chi_{8,9,10,q^3}^{a_8,a_9,a_{10}}$.  There are $(q-1)^3$ of these and they have degree $q^3$.
\item The characters $\chi_{8,9,10,\frac{q^3}{2}}^{a_8,a_9,a_{10},a_{5,6,7},d_{1,2,4},d_3}$.  There are
$4(q-1)^4$ of these and their degree is $\frac{q^3}{2}$.  The parameters $d_{1,2,4}$ and $d_3$ can each take one of two values as explained at the end of  \S\ref{ss:8910}.
\end{itemize}

Also for $p=2$ the parametrization of characters where there is a relation between
parameters involving the sum
of two terms  has to be updated in a similar way to the corresponding situation
for conjugacy class representatives.  For example, in the family $\cF_{5,6}$, we just have
characters $\chi_{5,6}^{a_5,a_6,b_2}$ (so we have set $b_1 = 0$ and $b_2$ ranges over $\F_q$).

\section{Determining the character values}
\label{sec:main}

In the following subsections we give the construction of
the irreducible characters and determine the character values.
We split up our determination of the character values according to the families given in Table \ref{tab:irrU},
and in each of these subsections we calculate the values of the irreducible characters
on elements of $U$.  For the families considered in \S\ref{ss:lin}, \S\ref{ss:567} and \S\ref{ss:89}
we determine the character values on general elements of $U$, whereas
for the other characters we take advantage of the representatives of the conjugacy classes
from Table \ref{tab:representatives}.

Before we embark on calculating the values of the irreducible
characters of $U$, we give some notation that we use.

Denote by $\Tr : \F_q \to \F_p$ the trace map, and define
$\phi : \F_q \to \C^\times$
by $\phi(x) = e^{\frac{i2\pi \Tr(x)}{p}}$, so that $\phi$ is a nontrivial character
from the additive group of $\F_q$ to the multiplicative group $\C^\times$.

We record two important elementary observations about $\phi$ that we require frequently in the sequel.
First we note that, for $t \in \F_q$, we have $\phi(st)=1$ for all $s \in \F_q$ if and only $t=0$,
second we note that
$\sum_{s \in \F_q} \phi(s) = 0$.

Let $H$ be a finite group.  We write $Z(H)$ for the centre of $H$.
Let $K$ be a subgroup of $H$, let $x \in H$, $h \in H$ and let $\psi : K \to \C$.  We write
$\dot \psi : H \to \C$ for the function defined by $\dot \psi(x) = \psi(x)$ if $x \in K$ and $\dot \psi(x) = 0$ if $x \notin K$.
We denote conjugation by $x^h = h^{-1}xh$ and ${}^h\psi : H \to \C$
is defined by ${}^h\psi(x) = \dot \psi(x^h)$.  If $\psi$ is a class function on $K$, then
we write $\psi^H$ for the induced class function on $H$, which is given by
the standard formula $\psi^H(h) = \frac{1}{|K|} \sum_{x \in H} \dot \psi(x^h)$.

We use the notation $x(\bt)$ from \eqref{e:x(t)}, where we recall that $x_3(t_3)$ is on the left.
Frequently, we consider quotients of $U$ by the normal subgroups $M_i$ defined in \eqref{e:Mi}.
When we do this we often want to identify the root subgroup $X_j$ with its image in $U/M_i$ for $j < i$, and
we will do this without further explanation.  We consider subgroups $V$ of $U/M_i$ of the form
$\prod_{j \in I} X_j$, where $I$ is a subset of $\{1,\dots,i\}$.  For such a subgroup $V$
we use the notation $x_V(\bt) = \prod_{j \in I} x_j(t_j)$ to denote a general element of $V$,
where the factors are ordered as in $x(\bt)$, so that $x_3(t_3)$ is on the left (if $3 \in I$).

We often use the notation $\delta_{\bs,0}$, where $\bs \in \F_q^m$ for some $m$, which
is defined as usual by $\delta_{\bs,0} = 1$ if $\bs = 0$ and $\delta_{\bs,0} = 0$ if $\bs \ne 0$.

\smallskip

As mentioned earlier, we have adapted the construction of the characters from \cite{HLM}
so that they are more explicit.  In particular, we realise each of the irreducible
characters as an induced character of a certain subgroup, which leads to a more
precise parametrization.  Given these differences, we briefly explain why our
constructions line up with those in \cite{HLM}.

For each irreducible character $\chi$ that we construct, we give a subgroup $V$ of a quotient $U/M$ of $U$ and a linear
character $\lambda$ of $V$.  Then $\chi$ is equal to the induced character $\lambda^{U/M}$ inflated to $U$ possibly tensored with a linear
character of $U$.
In each case there is a subgroup $X$ of $U/M$, which is a
transversal of $V$ in $U/M$.  In each of subsequent subsections, where we calculate the characters
values, we state what $V$ and $X$ are.  We also justify that the induced characters that we obtain
are indeed irreducible, and that the characters as given in Table \ref{tab:irrU} are distinct.
The main tool for achieving this is the Clifford theory for characters of special groups of type $q^{1+2m}$,
as is also the case in \cite{HLM}.
Now we can count the number of characters that we have constructed in each family, and verify
that this gives the same numbers in \cite{HLM}.  Therefore, we can deduce that we have
covered all of the irreducible characters of $U$.

\subsection{\texorpdfstring{The family $\mathcal{F}_{\lin}$}{The family F(lin)}} \label{ss:lin}

Here we state the value of $\chi_{\lin}^{b_1,b_2,b_3,b_4}$ on all elements of $U$.
We note that $M_5$ is the commutator subgroup of $U$, and the characters
$\chi_{\lin}^{b_1,b_2,b_3,b_4}$ are precisely the characters of $U/M_5$ inflated
to $U$; they are all the linear characters of $U$.  The following formula, therefore,
defines these characters.
$$
\chi_{\lin}^{b_1,b_2,b_3,b_4}(x(\bt)) =
\phi(b_1 t_1 + b_2t_2 + b_3t_3 + b_4t_4).
$$

\subsection{\texorpdfstring{The family $\cF_{5,6,7}^*$}{The family F(5,6,7)*}} \label{ss:567}
The characters in the families $\cF_5$, $\cF_6$, $\cF_7$, $\cF_{5,6}$, $\cF_{5,7}$, $\cF_{6,7}$ and $\cF_{5,6,7}$
can all be constructed in essentially the same way.
Thus we combine these families and denote their union by $\cF_{5,6,7}^*$.

We begin by giving the construction of the characters $\chi_{5,6,7}^{c_5,c_6,c_7,b_1,b_2,b_4}$, where
$c_5,c_6,c_7 \in \F_q$ are not all zero and $b_1,b_2,b_4 \in \F_q$.  From these
characters we obtain all of the characters in $\cF_{5,6,7}^*$.
First we note that $M_8$ is in the kernel of $\chi_{5,6,7}^{c_5,c_6,c_7,b_1,b_2,b_4}$.
As explained at the start of the section we identify $X_i$ with its image in $U/M_8$
for $i = 1,2,\dots,7$.

We have that $U/M_8 = X_3X_1X_2X_4X_5X_6X_7$, and that
$V=X_1X_2X_4X_5X_6X_7 \cong \F_q^6$ is an elementary abelian subgroup of $U/M_8$ of order $q^6$, and $X = X_3$ is a transversal of $U$ in $U/M_8$.
We define the linear character $\lambda^{c_5,c_6,c_7}$ of $V$  by $\lambda^{c_5,c_6,c_7}(x_V(\bt)) = \phi(c_5 t_5+c_6t_6+c_7t_7)$.

Let $\chi_{5,6,7}^{c_5,c_6,c_7}$ be the character of $U/M_8$ we obtain by inducing $\lambda^{c_5,c_6,c_7}$;
by a mild abuse of notation we also write $\chi_{5,6,7}^{c_5,c_6,c_7}$ for the character of $U$ given by inflation.
The character $\chi_{5,6,7}^{c_5,c_6,c_7,b_1,b_2,b_4}$ is defined by tensoring $\chi_{5,6,7}^{c_5,c_6,c_7}$ with the linear
character $\chi_{\lin}^{b_1,b_2,0,b_4}$.

The following proposition gives the character values of $\chi_{5,6,7}^{c_5,c_6,c_7,b_1,b_2,b_4}$.
It is clear from the calculations in the proof below that ${^x}\lambda^{c_5,c_6,c_7} \ne \lambda^{c_5,c_6,c_7}$ for all $x \in X$ with $x \ne 1$, so that
$\chi^{c_5,c_6,c_7,b_1,b_2,b_4}$ is irreducible by an application of Clifford theory.

\begin{proposition} \label{P:567}
$$\chi_{5,6,7}^{c_5,c_6,c_7,b_1,b_2,b_4}(x(\bt)) =  q \delta_{(t_3,c_5t_1+c_6t_2+c_7t_4),0} \phi(b_1t_1+b_2t_2+b_4t_4+c_5t_5+c_6t_6+c_7t_7)).$$
\end{proposition}

\begin{proof}
We write $\lambda$ for $\lambda^{c_5,c_6,c_7}$ and $\chi$ for $\chi_{5,6,7}^{c_5,c_6,c_7}$.
We have that $X = X_3$ is a transversal of $V$ in $U$, and also that
$Z(U/M_8) = X_5X_6X_7$.  From the induction formula we easily obtain that
$\chi(x(\bt))=0$ if $t_3 \neq 0$, and also that $\chi(x(\bt)) = \delta_{t_3,0} \chi(x_1(t_1)x_2(t_2)x_4(t_4))\lambda(x_5(t_5)x_6(t_6)x_7(t_7))$.
Thus it suffices to calculate $\chi(x_1(t_1)x_2(t_2)x_4(t_4))$.
We have
\begin{align*}
\chi(x_1(t_1)x_2(t_2)x_4(t_4)) &= \sum_{s \in \F_q} {}^{x_3(s)}\lambda(x_1(t_1)x_2(t_2)x_4(t_4)) \\
&= \sum_{s \in \F_q} \lambda((x_1(t_1)x_2(t_2)x_4(t_4))^{x_3(s)}) \\
&= \sum_{s \in \F_q} \lambda([x_1(t_1),x_3(s)][x_2(t_2),x_3(s)][x_4(t_4),x_3(s)]) \\
&= \sum_{s \in \F_q} \lambda(x_5(st_1)x_6(st_2)x_7(st_4)) \\
&= \sum_{s \in \F_q} \phi(s(c_5t_1+c_6t_2+c_7t_4)) \\
&= q \delta_{c_5t_1+c_6t_2+c_7t_4,0}.
\end{align*}
From this we can deduce the proposition.
\end{proof}

We note that the characters $\chi_{5,6,7}^{c_5,c_6,c_7,b_1,b_2,b_4}$ can be equal for different
values of $b_i$.  A fairly easy calculation shows that $\chi_{5,6,7}^{c_5,c_6,c_7,b_1,b_2,b_4} =
\chi_{5,6,7}^{c_5,c_6,c_7,b'_1,b'_2,b'_4}$ if and only if $(b_1-b'_1,b_2-b'_2,b_4-b'_4)$ is a
multiple of $(c_5,c_6,c_7)$.  Thus we choose representatives as follows for $p \ne 2,3$.
\begin{itemize}
\item If $c_5 \ne 0$ and $c_6,c_7 = 0$, then we take $b_1 = 0$.
\item If $c_5,c_6 \ne 0$ and $c_7 = 0$, then we take $b_1,b_2$ such that $c_6b_1+c_5b_2 = 0$.
\item If $c_5,c_6,c_7 \ne 0$, then we take $b_1,b_2,b_4$ such that $c_6c_7b_1+c_5c_7b_2+c_5c_6b_4 = 0$.
\end{itemize}
We deal with the other cases symmetrically with respect to the triality.
This gives the characters as in Table \ref{tab:irrU}, and we can give their character values from
Proposition \ref{P:567}.

We note that for $p=2$ or $p=3$ we have to choose our representatives
slightly differently.  For $p=3$ the only difference is when $c_5,c_6,c_7 \ne 0$
and this is shown in Table \ref{tab:irrU}.  In the case $p =2$, and we have $c_5,c_6 \ne 0$ and $c_7 = 0$,
then we take $b_1=0$, and we deal with the other cases symmetrically with respect to the triality.

\subsection{The Families \texorpdfstring{$\cF_8$}{F8}, \texorpdfstring{$\cF_9$}{F9}, \texorpdfstring{$\cF_{10}$}{F10},
\texorpdfstring{$\cF_{8,9}$}{F8,9}, \texorpdfstring{$\cF_{8,10}$}{F8,10} and \texorpdfstring{$\cF_{9,10}$}{F9,10}} \label{ss:89}
The characters in the families $\cF_8$, $\cF_9$, $\cF_{10}$, $\cF_{8,9}$, $\cF_{8,10}$
and $\cF_{9,10}$ are constructed in a similar way, so we deal with them
together.  We include the details for the characters in the family
$\cF_{8,9}$ and then remark that the calculations for the family
$\cF_8$ are entirely similar.
The character in the families $\cF_9$, $\cF_{10}$, $\cF_{8,10}$ and $\cF_{9,10}$
are defined from characters in $\cF_8$ or $\cF_{8,9}$ through the triality $\tau$.
Thus the values of these characters can be immediately deduced.

From Table \ref{tab:irrU} we see that the characters in $\cF_{8,9}$ fall into two
subfamilies determined by degree.  We start with the characters of degree $q^3$ and explain the construction
of the characters in this family.

First we note that $M_{10}$ lies in the kernel of the characters in $\cF_{8,9}$.
We have that $U/M_{10} = \prod_{i = 1}^9 X_i$ and we define the normal subgroup
$$
V = \prod_{i \ne 1,2,4} X_i
$$
of $U/M_{10}$.
We have that $V$ is isomorphic
to $\F_q^6$ and so is elementary abelian, and $X = X_1X_2X_4$ is a transversal of $V$ in $U/M_{10}$.  Therefore, we have linear
characters $\lambda^{a_8,a_9,a_6,a_7}$ of $V$ defined by
$$
\lambda^{a_8,a_9,a_6,a_7}(x_V(\bt))
= \phi(a_6t_6 + a_7 t_7 + a_8 t_8 + a_9 t_9),
$$
where $a_8,a_9,a_6,a_7 \in \F_q^\times$ and we assume that $(a_6,a_7)$ is not a multiple
of $(a_8,a_9)$.
Then we define $\chi_{8,9, q^3}^{a_8,a_9,a_6,a_7}$ to be the character we obtain by inducing  $\lambda^{a_8,a_9,a_6,a_7}$ to $U/M_{10}$.

From the calculations of the character values in Proposition \ref{P:89} below, we observe that
${^x}\lambda^{a_8,a_9,a_6,a_7} \ne \lambda^{a_8,a_9,a_6,a_7}$ for all $x \in X$ with $x \ne 1$..
Therefore, an application of Clifford theory implies that each $\chi_{8,9, q^3}^{a_8,a_9,a_6,a_7}$
is irreducible.

We proceed to consider the characters in $\cF_{8,9}$ of degree $q^2$.
Again we have that $M_{10}$ is
in the kernel of these characters.  We consider the subgroup
$$
W = \prod_{i \ne 1,5} X_i
$$
of $U/M_{10}$.
We note that $M = X_6X_7$ is a normal subgroup of $W$, that $W/M$ is
isomorphic to $\F_q^5$ and that $X = X_1X_5$ is a transversal of
$W$ in $U/M_{10}$.  Therefore, we have linear
characters $\mu^{a_8,a_9}$ of $W$ defined by
$$
\mu^{a_8,a_9}(x_W(\bt))
= \phi(a_8 t_8 + a_9 t_9),
$$
where $a_8,a_9 \in \F_q^\times$.
Then we define
$\chi_{8,9,q^2}^{a_8,a_9}$ to be the character we obtain by inducing  $\mu^{a_8,a_9}$ to $U/M_{10}$.
The characters
$\chi_{8,9,q^2}^{a_8,a_9,b_2,b_3,b_4}$ are given by tensoring
$\chi_{8,9,q^2}^{a_8,a_9}$ with $\chi_{\lin}^{0,b_2,b_3,b_4}$.

To see that $\chi_{8,9,q^2}^{a_8,a_9}$ is irreducible we can consider the restriction $\mu'$ of
$\mu^{a_8,a_9}$ to the subgroup $V' = X_2X_6X_8X_9$ of $V$.  Then we let
$\chi' = \mu'^{U'}$, where $U'$ is the subgroup $X_1X_2X_5X_6X_8X_9$
of $U/M_{10}$.  Now $V'$ is normal in $U'$ and it is an easy calculation to check
that ${^x}\mu' \ne \mu'$ for all $x \in X$ with $x \ne 1$.  Now
we deduce that $\chi'$ is irreducible by an application of Clifford theory.
From this it follows that $\chi_{8,9,q^2}^{a_8,a_9}$ is irreducible and thus also
$\chi_{8,9,q^2}^{a_8,a_9,b_2,b_3,b_4}$ is irreducible.

We now state our proposition giving the values of the characters in $\cF_{8,9}$.

\begin{proposition} \label{P:89}
{\em (a)}
$$
\chi_{8,9, q^3}^{a_8,a_9,a_6,a_7}(x(\bt)) = q^3 \delta_{(t_1,t_2,t_3,t_4,t_5,a_8t_6 + a_9t_7),0}
\phi(a_6t_6+ a_7 t_7 + a_8t_8+ a_9t_9).
$$
{\em (b)} If $t_3 = 0$, then
$$
\chi_{8,9,q^2}^{a_8,a_9,b_2,b_3,b_4}(x(\bt)) = q^2 \delta_{(t_1,t_5,a_8t_2+a_9t_4,a_8t_6+a_9t_7),0}
\phi( b_2 t_2 +b_4t_4 + a_8t_8 + a_9t_9).
$$
If $t_3 \ne 0$, then
\begin{multline*}
\chi_{8,9,q^2}^{a_8,a_9,b_2,b_3,b_4}(x(\bt)) = \\
q \delta_{(t_1,a_8t_2+a_9t_4),0}
\phi( b_2 t_2 + b_3t_3 +b_4t_4+ a_8t_8 +a_9t_9 + t_5(a_8t_2+a_9t_4 - (a_8t_6+a_9t_7)/t_3)).
\end{multline*}
\end{proposition}

\begin{proof}
(a) We write $\lambda = \lambda^{a_8,a_9,a_6,a_7}$ and $\chi = \chi_{8,9,q^2}^{a_8,a_9,a_6,a_7}$.
We consider the series
of normal subgroups of $U/M_{10}$
$$
V_0 = U/M_{10}, \quad V_1 = \prod_{i \ne 1} X_i, \quad V_2 = \prod_{i \ne
1,2} X_i, \quad V_3 = V.
$$
First we calculate the values of $\lambda^{V_2}$.  We have $Z(V_2) = X_6X_7X_8X_9$, so that
$$
\lambda^{V_2}(x_{V_2}(\bt)) = \lambda^{V_2}(x_3(t_3)x_4(t_4)x_5(t_5))\lambda(x_6(t_6)x_7(t_7)x_8(t_8)x_9(t_9)).
$$
Also $\lambda^{V_2}(x_{V_2}(\bt)) = 0$ if $t_4 \ne 0$.
Thus it suffices to
determine $\lambda^{V_2}(x_3(t_3)x_5(t_5))$.
Now for $s \in \F_q$, we have $(x_3(t_3)x_5(t_5))^{x_4(s)} = x_3(t_3)x_5(t_5)x_7(-st_3)x_9(-st_5)$
in $U/M_{10}$.
Therefore,
\begin{align*}
\lambda^{V_2}(x_3(t_3)x_5(t_5))
&= \sum_{s \in \F_q}  \lambda(x_3(t_3)x_5(t_5)x_7(-st_3)x_9(-st_5)) \\
&= \sum_{s \in \F_q}  \phi(-a_7st_3-a_9st_5) \\
&= \sum_{s \in \F_q}  \phi(-s(a_7t_3+a_9t_5)) \\
&= q \delta_{a_7t_3+a_9t_5,0}.
\end{align*}
and thus
$$
\lambda^{V_2}(x_{V_2}(\bt)) = q \delta_{(t_4,a_7t_3+a_9t_5),0}\phi(a_6t_6+a_7t_7+a_8t_8+a_9 t_9).
$$
Next we induce $\lambda^{V_2}$ to $V_1$; this requires a very similar calculation, which
we omit.  We obtain
$$
\lambda^{V_1}(x_{V_1}(\bt)) = q^2 \delta_{(t_2,t_4,a_7t_3+a_9t_5,a_6t_3+a_8t_5),0}\phi(a_6t_6+a_7t_7 + a_8t_8 + a_9 t_9).
$$
Our assumption that $(a_6,a_7)$ is not a multiple
of $(a_8,a_9)$ implies that for the above to be nonzero we require $t_3,t_5=0$.
Thus we deduce that
$$
\lambda^{V_1}(x_{V_1}(\bt)) = q^2 \delta_{(t_2,t_3,t_4,t_5),0}\phi(a_6t_6+a_7t_7 + a_8t_8 + a_9 t_9).
$$
Now we do the final induction up to $U/M_{10}$ to obtain $\chi$.  First we observe that
$Z(V_0) = X_8X_9$ and it is easy to check that
$\chi(x(\bt)) = 0$ if any of $t_1$, $t_2$, $t_3$, $t_4$ or $t_5$ is equal to $0$.
Thus it suffices to calculate
$\chi(x_6(t_6)x_7(t_7))$,
and we obtain
\begin{align*}
\chi(x_6(t_6)x_7(t_7)) &= \sum_{s \in \F_q} \lambda^{V_1}(x_6(t_6)x_7(t_7)x_8(-st_6)x_9(-st_7)) \\
&= q^2 \sum_{s \in \F_q} \phi(a_6t_7+a_7t_7 - a_8st_6 - a_9st_7) \\
&= q^3 \delta_{0,a_8t_6 + a_9t_7} \phi(a_6t_6+a_7t_7).
\end{align*}
Putting this all together we obtain the stated value of $\chi_{8,9, q^3}^{a_8,a_9,a_6,a_7}$.

(b) We write $\mu = \mu^{a_8,a_9}$ and $\chi = \chi_{8,9,q^2}^{a_8,a_9}$.
In order to determine the value of $\chi$ we first note that $Z(U/M_{10}) = X_8X_9$ and that
we have $\chi(x(\bt)) = 0$ if $t_1 \ne 0$.  Thus it suffices to calculate $\chi(x_3(t_3)x_2(t_2)x_4(t_4)x_5(t_5)x_6(t_6)x_7(t_7))$.
Next we calculate that
\begin{multline*}
(x_3(t_3)x_2(t_2)x_4(t_4)x_5(t_5)x_6(t_6)x_7(t_7))^{x_5(s_5)x_1(s_1)} =
\\
x_3(t_3)x_2(t_2)x_4(t_4)x_5(t_5-s_1t_3)x_6(t_6)x_7(t_7)x_8(-s_1t_6+s_5t_2+s_1t_2t_3)x_9(-s_1t_7+s_5t_4+s_1t_3t_4).
\end{multline*}
Thus, $\chi(x_3(t_3)x_2(t_2)x_4(t_4)x_5(t_5)x_6(t_6)x_7(t_7))$ is equal to
\begin{align*}
 & \: \sum_{s_1,s_5 \in \F_q} \mu(x_3(t_3)x_2(t_2)x_4(t_4)x_5(t_5)x_6(t_6)x_7(t_7))^{x_5(s_5)x_1(s_1)}) \\
= & \:  \sum_{s_1,s_5 \in \F_q} \delta_{t_5-s_1t_3,0} \phi(a_8(-s_1t_6+s_5t_2+s_1t_2t_3)+a_9(-s_1t_7+s_5t_4+s_1t_3t_4)).
\end{align*}
If $t_3 = 0$, then we see that this is equal to
$$
\delta_{t_5,0} \sum_{s_1,s_5 \in \F_q} \phi(a_8(-s_1t_6+s_5t_2)+a_9(-s_1t_7+s_5t_4)) = q^2 \delta_{(t_5,a_8t_2+a_9t_4,a_8t_6+a_9t_7),0}.
$$
Whereas for $t_3 \ne 0$, in the sum over $s_1$ we only get a nonzero contribution for $s_1 = t_5/t_3$, thus we get
\begin{multline*}
\sum_{s_5 \in \F_q} \phi(a_8(-t_5t_6/t_3+s_5t_2+t_2t_5)+a_9(-t_5t_7/t_3+s_5t_4+t_4t_5)) = \\
q \delta_{a_8t_2+a_9t_4,0}
\phi(t_5(a_8t_2+a_9t_4 - (a_8t_6+a_9t_7)/t_3)).
\end{multline*}
From the formulas above we deduce the stated values of $\chi_{8,9,q^2}^{a_8,a_9,b_2,b_3,b_4}$.
\end{proof}

We note the characters $\chi_{8,9, q^3}^{a_8,a_9,a_6,a_7}$ can be equal for different
values of $a_i$.  From the character values we observe that $\chi_{8,9, q^3}^{a_8,a_9,a_6,a_7} =
\chi_{8,9, q^3}^{a'_8,a'_9,a'_6,a'_7}$ if and only if $(a'_8,a'_9,a'_6,a'_7)$ is obtained from
$(a_8,a_9,a_6,a_7)$ by adding a multiple of $(a_8,a_9)$ to $(a_6,a_7)$.
So we choose $(a_6,a_7)$ such that $a_6a_9+a_7a_8=0$, except when $p=2$ where we just put $a_6=0$.

Also we observe that $\chi_{8,9,q^3}^{a_8,a_9,b_2,b_3,b_4} =
\chi_{8,9,q^3}^{a'_8,a'_9,b'_2,b'_3,b'_4}$ if and only if $(a'_8,a'_9,b'_2,b'_3,b'_4)$ is obtained from
$(a_8,a_9,b_2,b_3,b_4)$ by adding a multiple of $(a_8,a_9)$ to $(b_2,b_4)$.
So we choose a coset representative such that $b_2a_9+b_4a_8=0$ (except in the
case $p=2$ where we just take $b_2=0$).

We move on to briefly discuss the characters in the family $\cF_8$.  These can be defined in essentially the same
way as those in $\cF_{8,9}$.  The characters of degree $q^3$ are denoted $\chi_{8,q^3}^{a_8,a_7}$ and
are defined in exactly the same way as
$\chi_{8,9, q^3}^{a_8,a_9,a_6,a_7}$ except we set $a_9=0$ and $a_6=0$.  The characters $\chi_{8,q^2}^{a_8,b_3,b_4}$
are defined in the same way as $\chi_{8,9,q^2}^{a_8,a_9,b_2,b_3,b_4}$ except we set $a_9 = 0$ and $b_2=0$.
Thus we can immediately deduce the values of these characters given in the proposition below.

\begin{proposition} \label{P:8}
{\em (a)}
$$
\chi_{8,q^3}^{a_8,a_7}(x(\bt)) =  q^3 \delta_{(t_1,t_2,t_3,t_4,t_5,t_6),0}
\phi(a_7 t_7 + a_8t_8).
$$
{\em (b)}
If $t_3 = 0$, then
$$
\chi_{8,q^2}^{a_8,b_3,b_4}(x(\bt)) = q^2 \delta_{(t_1,t_2,t_5,t_6),0}
\phi(b_4t_4 + a_8t_8).
$$
If $t_3 \ne 0$, then
$$
\chi_{8,q^2}^{a_8,b_3,b_4}(x(\bt)) =  q \delta_{(t_1,t_2),0}
\phi(b_3t_3 +b_4t_4+ a_8(t_8-t_5t_6/t_3)).
$$
\end{proposition}

\subsection{The Family \texorpdfstring{$\cF_{8,9,10}$}{F8,9,10}} \label{ss:8910}
We begin with the assumption that $p > 2$, as the situation for $p=2$ is slightly
different and is discussed at the end of the section.

First we note that $M_{11}$ lies in the kernel of the characters in $\cF_{8,9,10}$.
We have that $U/M_{11} = \prod_{i = 1}^{10} X_i$ and we define the normal subgroup
$$
V = \prod_{i \ne 1,2,4} X_i
$$
of $U/M_{11}$.
We have that $V$ is isomorphic
to $\F_q^7$ and that $X = X_1X_2X_4$ is a transversal of $V$ in $U/M_{11}$. We define the
character $\lambda^{a_8,a_9,a_{10}}$ of $V$ by
$$
\lambda^{a_8,a_9,a_{10}}(x_V(\bt))
= \phi(a_8t_8+a_9t_9+a_{10} t_{10})
$$
for $a_8,a_9,a_{10} \in \F_q^\times$.
Then we define $\chi_{8,9,10}^{a_8,a_9,a_{10}}$ to be the character obtained by inducing
$\lambda^{a_8,a_9,a_{10}}$ to $U/M_{11}$.
The irreducible characters $\chi_{8,9,10}^{a_8,a_9,a_{10},b_3}$
in the family $\cF_{8,9,10}$
are then obtained by tensoring $\chi_{8,9,10}^{a_8,a_9,a_{10}}$ with the linear character $\chi_{\lin}^{0,0,b_3,0}$.
The construction of these characters is the same as in \cite[\S4]{HLM}, where they are shown to
be irreducible and distinct.

In order to obtain the values of $\chi$ we take advantage of the conjugacy class
representatives in Table \ref{tab:representatives}.  First note that
since $\chi(x(\bt)) = 0$ for $x(\bt) \not\in V$, so we do not need to consider
any conjugacy classes in Table \ref{tab:representatives} above the family $\cC_3$.
We split up our calculations for the remaining elements by considering elements in the family $\cC_3$,
and those lying in any of the other classes.

\begin{proposition} \label{P:8910}
Suppose $p > 2$.
\begin{itemize}
\item[(a)] Suppose that at least one of $t_1$, $t_2$ or $t_4$ is nonzero.  Then $\chi_{8,9,10}^{a_8,a_9,a_{10},b_3}(x(\bt)) = 0$.
\item[(b)] Let $x(\bt) = x_3(t_3)x_8(t_8)x_9(t_9)x_{10}(t_{10})$, where $t_3 \ne 0$.  Then
$$
\chi_{8,9,10}^{a_8,a_9,a_{10},b_3}(x(\bt)) = q \phi( b_3t_3 + a_8 t_8 + a_9t_9 + a_{10}t_{10})\sum_{s \in \F_q} \phi(-a_8a_9a_{10}t_3s^2) .
$$
\item[(c)]  Let $x(\bt) = x_5(t_5)x_6(t_6)x_7(t_7)x_8(t_8)x_9(t_9)x_{10}(t_{10})$.  Then
$$
\chi_{8,9,10}^{a_8,a_9,a_{10},b_3}(x(\bt)) = q^3 \delta_{(t_5,t_6,t_7),0} \phi( a_8 t_8 + a_9t_9 + a_{10}t_{10}).
$$
\end{itemize}
\end{proposition}

\begin{remark}
Before we proceed to the proof, we remark on the term $\sum_{t \in \F_q} \phi(-a_8a_9a_{10}t_3t^2)$ in (b) (for the case $t_3 \ne 0$). Clearly this can be deduced from
the value of $\sum_{t \in \F_q} \phi(t^2)$: if $-a_8a_9a_{10}t_3 \in \F_q^\times$ is a square, then $\sum_{t \in \F_q} \phi(-a_8a_9a_{10}t_3t^2) = \sum_{t \in \F_q} \phi(t^2)$ and if $-a_8a_9a_{10}t_3 \in \F_q^\times$ is not a square, then $\sum_{t \in \F_q} \phi(-a_8a_9a_{10}t_3t^2) = -\sum_{t \in \F_q} \phi(t^2)$.
  The sum $\sum_{t \in \F_q} \phi(t^2)$
is referred to as a {\em quadratic Gauss sum}.  For $q = p$ the value of this Gauss sum is well known: if $p \equiv 1 \bmod 4$, then $\sum_{t \in \F_p} \phi(t^2) = \sqrt p$, and if $p \equiv -1 \bmod 4$, then $\sum_{t \in \F_p} \phi(t^2) = i \sqrt p$.
However, for general $q$ the situation is more complicated: it is known that the absolute value is $\sqrt q$,
see for example \cite[Proposition 11.5]{IK}, but a formula giving the exact value appears not to be known in general.
\end{remark}

\begin{proof}
We begin by noting that it suffices to consider the case $b_3 =0$.  We write $\lambda = \lambda^{a_8,a_9,a_{10}}$ and
$\chi = \chi_{8,9,10}^{a_8,a_9,a_{10}}$.

(a) We dealt with this case in the discussion before the statement of the proposition.

(b) The first step is to calculate in $U/M_{11}$ that
$$
x_3(t_3)^{x_1(s_1)x_2(s_2)x_4(s_4)} =
x_3(t_3)x_5(-s_1t_3)x_6(-s_2t_3)x_7(-s_4t_3)x_8(s_1s_2t_3)x_9(s_1s_4t_3)x_{10}(s_2s_4t_3).
$$
Therefore, we have
\begin{align*}
\chi(x_3(t_3)) &= \sum_{s_1,s_2,s_4\in\F_q}\phi(a_8s_1s_2t_3+a_9s_1s_4t_3+a_{10}s_2s_4t_3)\\
&= \sum_{s_2,s_4\in\F_q}\phi(a_{10}s_2s_4t_3)\sum_{s_1\in\F_q}\phi(s_1(a_8s_2+a_9s_4)t_3)\\
&=\sum_{s_2,s_4\in\F_q}\phi(a_{10}t_3s_2s_4)q\delta_{0,a_8s_2 +a_9s_4}\\
&=q\sum_{s_4\in\F_q}\phi(-a_{10}a_9s_4^2t_3/a_8) \\
&=q\sum_{s\in\F_q}\phi(-a_8a_9a_{10}t_3s^2).
\end{align*}
For the last equality we just substitute $s_4 = a_8s$.
Now using that $Z(U/M_{11}) = X_8X_9X_{10}$, the value of $\chi$ given in (b) can be deduced.

(c)  In this case the starting point is to calculate in $U/M_{11}$ that
\begin{multline*}
(x_5(t_5)x_6(t_6)x_7(t_7))^{x_1(s_1)x_2(s_2)x_4(s_4)} =
\\
x_5(t_5)x_6(t_6)x_7(t_7)x_8(-s_1t_6-s_2t_5)x_9(-s_1t_7-s_4t_5)x_{10}(-s_2t_7-s_4t_6).
\end{multline*}
Therefore, $\chi(x_5(t_5)x_6(t_6)x_7(t_7))$ is equal to
\begin{align*}
 & \: \sum_{s_1\in\F_q}\phi(-s_1(a_8t_6+a_9t_7))\sum_{s_2\in\F_q}\phi(-s_2(a_8t_5+a_{10}t_7))\sum_{s_4\in\F_q}\phi(-s_4(a_9t_5+a_{10}t_6))\\
= & \:  q^3 \delta_{(a_8t_6+a_9t_7,a_8t_5+a_{10}t_7,a_9t_5+a_{10}t_6),0}.
\end{align*}
Now the linear system
$$
\begin{array}{llllllll} & &a_8t_6&+&a_9t_7&=&0\\a_8t_5&&&+&a_{10}t_7&=&0\\a_9t_5&+&a_{10}t_6&&&=&0\end{array}
$$
has only the trivial solution $t_5=t_6=t_7=0$, as we are assuming $p > 2$.
Thus we get $\chi(x_5(t_5)x_6(t_6)x_7(t_7)) =  q^3 \delta_{(t_5,t_6,t_7),0}$.
Now using that $Z(U/M_{11}) = X_8X_9X_{10}$, the value of $\chi_{8,9,10}^{a_8,a_9,a_{10},b_3}$ given in (c) can be deduced.
\end{proof}

To end this section we consider the case $p=2$, where the situation is
more complicated.  The characters $\chi_{8,9,10,q^3}^{a_8,a_9,a_{10}}$ are defined in
the same way as the characters $\chi_{8,9,10}^{a_8,a_9,a_{10}}$ for $p>3$.
There are also characters denoted $\chi_{8,9,10,q^3/2}^{a_8,a_9,a_{10},a_{5,6,7},d_{1,2,4},d_3}$, which are
defined as follows.

Again we let $V = \prod_{i \ne 1,2,4} X_i$ and let $a = a_{5,6,7},a_8,a_9,a_{10} \in \F_q^\times$.
We define the linear character
$\lambda = \lambda^{a_8,a_9,a_{10},a}$ of $V$ by
$$
\lambda(x_V(\bt)) = \phi(a_8a_9at_5+a_8a_{10}at_6+a_9a_{10}at_7+a_8t_8+a_9t_9+a_{10}t_{10}).
$$
Now let $Y_{124}=\{1,x_1(a_{10}a)x_2(a_9a)x_4(a_8a)\}$
and $\bar V = Y_{124}V$.    We can  extend $\lambda$ to $\bar V$ in two ways, namely $\bar \lambda^{d_{1,2,4}}$,
where $\bar \lambda^{d_{1,2,4}}(x_1(a_{10}a)x_2(a_9a)x_4(a_8a)) = (-1)^{d_{1,2,4}}$
for $d_{1,2,4} \in \F_2$.
We then induce $\bar \lambda^{d_{1,2,4}}$ to $U/M_{11}$ to obtain a character $\psi^{d_{1,2,4}}$, which is
irreducible as shown in \cite[\S4]{HLM}.
Now let $d_3$ be either 0 or an element of $\F_q$, which is not in the image of the map $s \mapsto aa_8a_9a_{10}s+a_8a_9a_{10}s^2$.
Then we define $\chi_{8,9,10,q^3/2}^{a_8,a_9,a_{10},a_{5,6,7},d_{1,2,4},d_3}$ by tensoring
$\psi^{d_{1,2,4}}$ with $\chi_{\lin}^{0,0,d_3,0}$.

The proposition below states the values of these characters, these can be deduced from \cite[Thm.\ 2.3]{LM}.
Note that with our specific choice of $\phi$,
we have $\ker(\phi)=\{t^p-t:t\in\F_q\}$. Thus for each $a\in\F_q^\times$, we have that
$\ker(\phi)=a^{-p}\T_a$, where $\T_a =\{t^p-a^{p-1}t \mid t\in\F_q\}$ as defined in \cite[Definition 1.2]{LM}.
In turn this implies that $a_\phi$ as defined in \cite[Definition 1.4]{LM}
is equal to $a^{-p}$.
We note that the values for $\chi_{8,9,10,q^3}^{a_8,a_9,a_{10}}$ look different
to those for $p > 2$, which is explained by the fact that the quadratic Gauss sums
for $p=2$ are clearly zero.

\begin{proposition}
\label{p:8910p2}
Let $p = 2$.

\noindent
{\em (a)}
$$
\chi_{8,9,10,q^3}^{a_8,a_9,a_{10}}(x(\bt))=\delta_{(t_1,t_2,t_3,t_4,a_8t_5+a_{10}t_7,a_8t_6+a_9t_7),0}q^3\phi(a_8t_8+a_9t_9+a_{10}t_{10}).
$$
{\em (b)} Let $a = a_{5,6,7}$.
If $(t_1,t_2,t_4) \not\in \{(0,0,0),(a_{10}a,a_9a,a_8a)\}$, then
$$
\chi_{8,9,10,q^3/2}^{a_8,a_9,a_{10},a_{5,6,7},d_{1,2,4},d_3}(x(\bt)) = 0.
$$
If $x(\bt) = x_1(a_{10}r)x_2(a_9r)x_4(a_8r)x_5(a_{10}t)x_6(a_9t)x_7(a_8t)x_8(t_8)x_9(t_9)x_{10}(t_{10})$,
where $r \in \{0,a\}$ and $t,t_8,t_9,t_{10} \in \F_q$, then
$$
\chi_{8,9,10,q^3/2}^{a_8,a_9,a_{10},a_{5,6,7},d_{1,2,4},d_3}(x(\bt)) = \frac{q^3}{2}\phi(d_{1,2,4}r+a_8a_9a_{10}at+a_8t_8+a_9t_9+a_{10}t_{10}).
$$
For $r \in \{0,a\}$ and $x(\bt) = x_3(t_3)x_1(a_{10}r)x_2(a_9r)x_4(a_8r)x_5(t_5)x_6(t_6)x_7(t_7)x_8(t_8)x_9(t_9)x_{10}(t_{10})$,
\begin{multline*}
\chi_{8,9,10,q^3/2}^{a_8,a_9,a_{10},a_{5,6,7},d_{1,2,4},d_3}(x(\bt)) = \frac{q^2}{2}\delta_{t_3,\frac{a^{-p}}{a_8a_9a_{10}}}\phi(d_{1,2,4}r+d_3t_3+a_8a_9a_{10}a\frac{t_7}{a_8}+ \\
\frac{(a_8a_9a_{10})^2}{a^{-p}}(\frac{t_5}{a_{10}}+\frac{t_7}{a_8})(\frac{t_6}{a_9}+\frac{t_7}{a_8})+a_8t_8+a_9t_9+a_{10}t_{10}).
\end{multline*}
\end{proposition}

\subsection{The Family \texorpdfstring{$\cF_{11}$}{F11}}
First we note that $M_{12}$ lies in the kernel of the characters in $\cF_{11}$.
We have that $U/M_{12} = \prod_{i = 1}^{11} X_i$ and we define the normal subgroup
$$
V = \prod_{i \ne 1,2,4} X_i
$$
of $U/M_{12}$.
We have that $V$ is isomorphic
to $\F_q^8$ and $X = X_1X_2X_4$ is a transversal of $V$ in $U/M_{12}$.
We define the
character $\lambda^{a_{11},b_5,b_6,b_7}$ of $V$ by
$$
\lambda^{a_{11},b_5,b_6,b_7}(x_V(\bt))
= \phi(b_5t_5+b_6t_6+b_7t_7+a_{11} t_{11})
$$
for $a_{11} \in \F_q^\times$ and $b_5,b_6,b_7 \in \F_q$.

We define $\chi_{11}^{a_{11},b_5,b_6,b_7}$ to be the character obtained by inducing
$\lambda^{a_{11},b_5,b_6,b_7}$ to $U/M_{12}$.
The irreducible characters $\chi_{11}^{a_{11},b_5,b_6,b_7,b_3}$
in the family $\cF_{11}$
are then obtained by tensoring $\chi_{11}^{a_{11},b_5,b_6,b_7}$ with the
linear character $\chi_{\lin}^{0,0,b_3,0}$.

Using \cite[Lemma 1.5]{LM} with $Z=X_{11}$, $Y=X_8X_9X_{10}$, $X=X_1X_2X_4$ and $M=X_3X_5X_6X_7$,
we obtain that each $\chi_{11}^{a_{11},b_5,b_6,b_7,b_3}$ is irreducible, and that  $\chi_{11}^{a_{11},b_5,b_6,b_7,b_3}
= \chi_{11}^{a'_{11},b'_5,b'_6,b'_7,b'_3}$
if and only if $(a_{11},b_5,b_6,b_7,b_3) = (a'_{11},b'_5,b'_6,b'_7,b'_3)$.
The fact that the hypothesis of that lemma holds follows from the fact that $XYZ$ forms a subgroup of $V$
which is special of type $q^{1+6}$.

In order to obtain the values of $\chi_{11}^{a_{11},b_5,b_6,b_7,b_3}$ in Proposition \ref{P:11} below, we take advantage of the conjugacy class
representatives in Table \ref{tab:representatives}.
We assume $p>2$ at first, as
the situation is slightly different for $p=2$, which we discuss after.
First note that
since $\chi_{11}^{a_{11},b_5,b_6,b_7,b_3}(x(\bt)) = 0$ for $x(\bt) \not\in V$, we can deal with
any conjugacy classes in Table \ref{tab:representatives} above the family $\cC_3$ quickly in (a).
Dealing with elements in the family $\cC_3$ in (b) is the most complicated part of the proposition.  Part (c) covers all
of those in the families $\cC_{5,6,7}$, $\cC_{5,6}$ and $\cC_{5}$, and then we can
then use the triality automorphism $\tau$ to determine the values
on representatives in the classes $\cC_{5,7}$, $\cC_{6,7}$, $\cC_{6}$ and $\cC_{7}$.
We finish by considering the remaining classes together in (d), so that we have
determined the value of $\chi_{11}^{a_{11},b_5,b_6,b_7,b_3}$.

\begin{proposition} \label{P:11}
Suppose $p > 2$. \\
{\em (a)} Suppose that at least one of $t_1$, $t_2$ or $t_4$ is nonzero.  Then
$$
\chi_{11}^{a_{11},b_5,b_6,b_7,b_3}(x(\bt)) = 0.
$$
{\em (b)} Let $x(\bt) = x_3(t_3)x_8(t_8)x_9(t_9)x_{10}(t_{10})x_{11}(t_{11})$, where $t_3 \ne 0$.  Then
\begin{multline*}
\chi_{11}^{a_{11},b_5,b_6,b_7,b_3}(x(\bt)) =  q\phi(b_3t_3+a_{11}t_{11})(q\delta_{(b_6t_3+a_{11}t_9,b_7t_3+a_{11}t_8),0}) \\ \sum_{s\in\F_q^\times}\phi(-(b_5t_3+a_{11}t_{10})s + (b_6t_3+a_{11}t_9)(b_7t_3+a_{11}t_8)(a_{11}t_3)^{-1}s^{-1})).
\end{multline*}
{\em (c)}
\begin{itemize}
\item[(i)] Let $x(\bt) = x_5(t_5)x_6(t_6)x_7(t_7)x_{11}(t_{11})$, where $t_5,t_6,t_7 \ne 0$.  Then
$$
\chi_{11}^{a_{11},b_5,b_6,b_7,b_3}(x(\bt)) = q\phi(b_5t_5+b_6t_6+b_7t_7+a_{11}t_{11})\sum_{s \in \F_q} \phi(-a_{11}t_5t_6t_7s^2).
$$
\item[(ii)] Let $x(\bt) = x_5(t_5)x_6(t_6)x_{9}(t_9)x_{10}(t_{10})x_{11}(t_{11})$, where $t_5 \ne 0$. Then
$$
\chi_{11}^{a_{11},b_5,b_6,b_7,b_3}(x(\bt)) =  q^2\delta_{t_6t_9,t_5t_{10}}\ \phi(b_5t_5+b_6t_6+a_{11}t_{11}).
$$
\end{itemize}
{\em (d)} Let $x(\bt) = x_8(t_8)x_{9}(t_9)x_{10}(t_{10})x_{11}(t_{11})$. Then
$$
\chi_{11}^{a_{11},b_5,b_6,b_7,b_3} =  q^3\delta_{(t_8,t_9,t_{10}),0}\phi(a_{11}t_{11}).
$$
\end{proposition}

\begin{remark}
Before the proof, we comment on the term of the form
$\sum_{s \in \F_q} \phi(As+\frac{B}{s})$, where $A,B \in \F_q^\times$ in (b).
It seems that a general formula for this is not known.
\end{remark}

\begin{proof}
We note that it suffices to deal with the case $b_3=0$ and we
let $\lambda = \lambda^{a_{11},b_5,b_6,b_7}$ and $\chi = \chi^{a_{11},b_5,b_6,b_7}$.

(a) We dealt with this case in the discussion before the statement of the proposition.

(b) First we note that $Z(U/M_{12}) = X_{11}$.  Therefore, it suffices to calculate the values of
$\chi(x_3(t_3)x_8(t_8)x_9(t_9)x_{10}(t_{10}))$.  We calculate
\begin{multline*}
(x_3(t_3)x_8(t_8)x_9(t_9)x_{10}(t_{10}))^{x_1(s_1)x_2(s_2)x_4(s_4)} = \\
x_3(t_3)x_5(-t_3s_1)x_6(-t_3s_2)x_7(-t_3s_4) x_8(t_8+t_3s_1s_2) \\
x_9(t_9+t_3s_1s_4)x_{10}(t_{10}+t_3s_2s_4)
x_{11}(-t_9s_2-t_{10}s_1-t_8s_4-t_3s_1s_2s_4)).
\end{multline*}
Thus we see that $\chi(x_3(t_3)x_8(t_8)x_9(t_9)x_{10}(t_{10}))$ is equal to
\begin{align*}
& \: \sum_{s_1,s_2,s_4\in\F_q}\phi(-b_5t_3s_1-b_6t_3s_2-b_7t_3s_4+a_{11}(-t_9s_2-t_{10}s_1-t_8s_4-t_3s_1s_2s_4)) \\
=& \: \sum_{s_1,s_2\in\F_q}\phi(-b_5t_3s_1-b_6t_3s_2+a_{11}(-t_9s_2-t_{10}s_1)) \sum_{s_4\in\F_q}\phi(-(b_7t_3+a_{11}t_8+a_{11}t_3s_1s_2)s_4) \\
=& \: \sum_{s_1,s_2\in\F_q}\phi(-b_5t_3s_1-b_6t_3s_2+a_{11}(-t_9s_2-t_{10}s_1)) q\delta_{b_7t_3+a_{11}t_8+a_{11}t_3s_1s_2,0}\\
=& \: q \left(\sum_{s_2\in\F_q}\phi((-b_6t_3-a_{11}t_9)s_2) \delta_{b_7t_3+a_{11}t_8,0} \right.\\
& \: \left.+\sum_{s_1\in\F_q^\times}\phi(-b_5t_3s_1-a_{11}t_{10}s_1)\sum_{s_2\in\F_q}\phi((-b_6t_3-a_{11}t_9)s_2)
\delta_{b_7t_3+a_{11}t_8+a_{11}t_3s_1s_2,0}\right) \\
=& \:q\left( \phantom{\sum_{a}}\hspace{-0.2in} q\delta_{(b_7t_3+a_{11}t_8,b_6t_3+a_{11}t_9),0} \right.\\
& \: \left. + \sum_{s_1\in\F_q^\times}\phi(-(b_5t_3+a_{11}t_{10})s_1)
\phi((b_6t_3+a_{11}t_9)(b_7t_3+a_{11}t_8)(a_{11}t_3)^{-1}s_1^{-1})\right) \\
=& \: q(q\delta_{(b_7t_3+a_{11}t_8,b_6t_3+a_{11}t_9),0} \\
& \:  + \sum_{s \in\F_q^\times}\phi(-(b_5t_3+a_{11}t_{10})s+(b_6t_3+a_{11}t_9)(b_7t_3+a_{11}t_8)(a_{11}t_3)^{-1}s^{-1})).
\end{align*}
This gives the formula in (b).

(c)
We move on to consider the case where $t_3 = 0$ and $t_5 \ne 0$.  Consider the normal subgroup
$W= \prod_{i \ne 1} X_i$ of $U/M_{12}$.  A calculation shows that
\begin{multline*}
\lambda^W(x_5(t_5)x_6(t_6)x_7(t_7)x_8(t_8)x_9(t_9)x_{10}(t_{10})x_{11}(t_{11})) = \\
q\phi(b_5t_5+b_6t_6+b_7t_7+a_{11}(-t_5^{-1}t_8t_9+t_{11})).
\end{multline*}
To obtain $\chi(x_5(t_5)x_6(t_6)x_7(t_7)x_8(t_8)x_9(t_9)x_{10}(t_{10})x_{11}(t_{11}))$
we just have to induce this formula over $X_1$ to $U/M_{12}$.
First we consider
elements of the conjugacy class $\cC_{5,6,7}$, and we get
\begin{align*}
& \: \chi(x_5(t_5)x_6(t_6)x_7(t_7)x_{11}(t_{11})) \\
= & \: q\phi(b_5t_5+b_6t_6+b_7t_7+a_{11}t_{11})\sum_{s_1 \in \F_q} \phi(-a_{11}t_5^{-1}(-s_1t_6)(-s_1t_7)) \\
= & \: q\phi(b_5t_5+b_6t_6+b_7t_7+a_{11}t_{11})\sum_{s \in \F_q} \phi(-a_{11}t_5t_6t_7s^2)
\end{align*}
as stated in (i).

Next we consider elements in the class $\cC_{5,6}$ or $\cC_5$ and we get
\begin{align*}
& \: \chi(x_5(t_5)x_6(t_6)x_{9}(t_9)x_{10}(t_{10})x_{11}(t_{11})) \\
= & \: q\phi(b_5t_5+b_6t_6+a_{11}t_{11})\sum_{s_1 \in \F_q} \phi(-a_{11}t_5^{-1}((-s_1t_6)t_9-s_1t_{10})) \\
= & \:  q\phi(b_5t_5+b_6t_6+b_7t_7+a_{11}t_{11})\sum_{s \in \F_q} \phi(a_{11}t_5^{-1}s_1(t_6t_9+t_{10})) \\
= & \: q^2\delta_{t_6t_9+t_{10},0}\phi(b_5t_5+b_6t_6+b_7t_7+a_{11}t_{11})
\end{align*}
as stated in (ii).

(d) This is an easy calculation, which we omit.
\end{proof}

In the case $p=2$,
the only difference is that we now have to consider elements of the form $x_5(t_5)x_6(t_6)x_7(t_7)x_{10}(t_{10})x_{11}(t_{11})$ even when
$t_5,t_6,t_7 \ne 0$, to cover the family of conjugacy classes $\cC_{5,6,7,2q^8}^{p=2}$;
the value given below is in a form where we do not have a closed formula.  Also the value for the family $\cC_{5,6,7,q^8}^{p=2}$
is simplified as the quadratic Gauss sum is equal to 0.

\begin{lemma} \label{l:11p2}
Let $p=2$. \\
{\em (a)} Let $x(\bt) = x_5(t_5)x_6(t_6)x_7(t_7)x_{10}(t_{10})x_{11}(t_{11})$, where $t_5,t_6,t_7,t_{10} \in \F_q^\times$ and $t_{11} \in \F_q$.
Then
$$
\chi_{11}^{a_{11},b_5,b_6,b_7,b_3}(x(\bt)) = q\phi(b_5t_5+b_6t_6+b_7t_7+a_{11}t_{11})\sum_{s \in \F_q} \phi(a_{11}(t_5t_6t_7s^2+t_5t_{10}s)).
$$
{\em (b)} Let $x(\bt) = x_5(t_5)x_6(t_6)x_7(t_7)$, where $t_5,t_6,t_7 \in \F_q^\times$.
Then
$$
\chi_{11}^{a_{11},b_5,b_6,b_7,b_3}(x(\bt)) = 0.
$$
\end{lemma}

\begin{proof}
(a) We can proceed as in the case $p>2$ and we obtain the formula
\begin{multline*}
\lambda^W(x_5(t_5)x_6(t_6)x_7(t_7)x_8(t_8)x_9(t_9)x_{10}(t_{10})x_{11}(t_{11})) = \\
q\phi(b_5t_5+b_6t_6+b_7t_7+a_{11}(t_5^{-1}t_8t_9+t_{11})).
\end{multline*}
Now we induce over $X_1$ to obtain
\begin{align*}
& \hspace{-2cm} \chi(x_5(t_5)x_6(t_6)x_7(t_7)x_{10}(t_{10})x_{11}(t_{11})) = \\
 & \: \sum_{s_1 \in \F_q} \lambda^W(x_5(t_5)x_6(t_6)x_7(t_7)x_8(s_1t_6)x_9(s_1t_7)x_{10}(t_{10})x_{11}(s_1t_{10}+t_{11})) \\
= & \: \sum_{s_1 \in \F_q} q\phi(b_5t_5+b_6t_6+b_7t_7+a_{11}(t_5^{-1}s_1^2t_6t_7+s_1t_{10}+t_{11}) \\
= & \: q\phi(b_5t_5+b_6t_6+b_7t_7 + a_{11}t_{11})\sum_{s \in \F_q} \phi(a_{11}(t_5t_6t_7s^2+t_5t_{10}s)).
\end{align*}

(b) This is obtained directly from the calculation for $p > 2$ by noting that the sum there is equal to 0.
\end{proof}

\subsection{The Family \texorpdfstring{$\cF_{12}$}{F12}}
Let
$V =
\prod_{i \ne 3,5,6,7} X_i$ and $W= X_8X_9X_{10}X_{11}$.
Then $W$ is a normal subgroup of $V$ and the
quotient $V/W$ is elementary abelian isomorphic to $\F_q^4$.  Also
$X = X_3X_5X_6X_7$ is a transversal of $V$ in $U$.
Given $a_{12} \in \F_q^\times$, we define the linear character
$\lambda^{a_{12}} : V \to \C$ by $\lambda^{a_{12}}(x_{V}(\bt))= \phi(a_{12}t_{12})$.  Then we define
$\chi_{12}^{a_{12}} = (\lambda^{a_{12}})^U$ and
$\chi_{12}^{a_{12},b_1,b_2,b_4} = \chi_{12}^{a_{12}}
\chi_{\lin}^{b_1,b_2,0,b_4}$.

Using \cite[Lemma 1.5]{LM} with $Z=X_{12}$, $Y=X_8X_9X_{10}X_{11}$, $X=X_3X_5X_6X_7$ and $M=X_1X_2X_4$,
we obtain that each $\chi_{11}^{a_{12},b_1,b_2,b_4}$ is irreducible, and that $\chi_{11}^{a_{12},b_1,b_2,b_4} = \chi_{11}^{a'_{12},b'_1,b'_2,b'_4}$
if and only if $(a_{12},b_1,b_2,b_4) = (a'_{12},b'_1,b'_2,b'_4)$.
The hypothesis of that lemma holds due to the fact that $XYZ$ forms a subgroup of $V$
which is special of type $q^{1+8}$.

In the next proposition we state the values of $\chi_{12}^{a_{12},b_1,b_2,b_4}$
by considering a number of cases, which cover all of the conjugacy class representatives
in Table \ref{tab:representatives}.  We note that in (c) we only deal with
conjugacy classes up to the action of $\tau$.  We first restrict to the case $p > 2$, and then afterwards we
describe some differences that occur for the case $p=2$.

\begin{proposition} \label{P:12}
Suppose that $p > 2$.

\noindent
{\em (a)} Suppose that $t_3 \ne 0$.  Then
$$
\chi_{12}^{a_{12},b_1,b_2,b_4}(x(\bt)) = 0.
$$

\noindent
{\em (b)} Suppose $x(\bt)$ lies in one of the families $\cC_{1,2,4,q^6}$,
$\cC_{1,2,q^6}$, $\cC_{1,4,q^6}$, $\cC_{2,4,q^6}$, $\cC_{1,q^6}$,
$\cC_{2,q^6}$ or $\cC_{4,q^6}$. Then
$$
\chi_{12}^{a_{12},b_1,b_2,b_4}(x(\bt)) = 0.
$$

\noindent
{\em (c)} The values of $\chi_{12}^{a_{12},b_1,b_2,b_4}$ on the families of conjugacy classes $\cC_{1,2,4,q^7}$, $\cC_{1,2,q^7}$,
$\cC_{1,2,q^8}$, $\cC_{1,q^7}$ and $\cC_{1,q^8}$ are given as follows.
\begin{enumerate}
\item[$\cC_{1,2,4,q^7}$:]\ $\chi_{12}^{a_{12},b_1,b_2,b_4}(x_1(t_1)x_2(t_2)x_4(t_4)x_{12}(t_{12}))$ \\
$\displaystyle = q\phi(b_1t_1+b_2t_2+b_4t_4+a_{12}t_{12})\sum_{s \in \F_q} \phi(- a_{12}t_1t_2
t_4s^2)$, for $t_1,t_2,t_4 \in \F_q^\times$ and $t_{12} \in \F_q$.
\item[$\cC_{1,2,q^7}$:]\ $\chi_{12}^{a_{12},b_1,b_2,b_4}(x_1(t_1)x_2(t_2)x_9(t_9)x_{10}(t_{10}))
= q^2 \delta_{0,t_{10}{t_1} -  t_2 t_9}\phi(b_1t_1+b_2t_2)$, \\
for $t_1,t_2,t_9,t_{10} \in \F_q^\times$.
\item[$\cC_{1,2,q^8}$:]\ $\chi_{12}^{a_{12},b_1,b_2,b_4}(x_1(t_1)x_2(t_2)x_{12}(t_{12}))
= q^2 \phi(b_1t_1+b_2t_2+a_{12}t_{12})$, \\
for $t_1,t_2 \in \F_q^\times$ and $t_{12} \in \F_q$.
\item[$\cC_{1,q^7}$:]\ $\chi_{12}^{a_{12},b_1,b_2,b_4}(x_1(t_1)x_{10}(t_{10}))
= 0$, for $t_1,t_{10} \in \F_q^\times$.
\item[$\cC_{1,q^8}$:]\ $\chi_{12}^{a_{12},b_1,b_2,b_4}(x_1(t_1)x_{12}(t_{12}))
= q^2 \phi(b_1t_1+a_{12}t_{12})$, for $t_1\in \F_q^\times$ and $t_{12} \in \F_q$.
\end{enumerate}

\noindent
{\em (d)} Let $x(\bt) = \prod_{i \ge 5} x_i(t_i)$.  Then
$$\chi_{12}^{a_{12},b_1,b_2,b_4}(x(\bt)) = q^4
\delta_{0,(t_5,t_6,t_7,t_8,t_9,t_{10},t_{11})} \phi(a_{12}t_{12}).
$$
\end{proposition}

\begin{proof}
It is clear that it suffices to
calculate the values of $\chi_{12}^{a_{12}}$.  We write
$\lambda = \lambda^{a_{12}}$ and $\chi = \chi_{12}^{a_{12}}$.

In order to explain our calculations, we introduce some intermediate
subgroups.  We let $V_1 = VX_7$, $V_2 = VX_6X_7$ and
$V_3 = VX_5X_6X_7$.  To abbreviate notation
we write $\lambda_i =
\lambda^{V_i}$ for each $i=1,2,3$.
By direct calculation we find
$$
\lambda_1(x_{V_1}(\bt))
= q\delta_{(t_7,t_8),0}  \phi(a_{12} t_{12}),$$
and
\begin{equation}\label{e:lambda2}
\lambda_2(x_{V_2}(\bt))
= \left\{ \begin{array}{ll} q \delta_{0,(t_6,t_7)} \phi(a_{12}(t_{12} - \frac{t_9t_8}{t_1})) & \text{if $t_1 \ne 0$}
\\ q^2 \delta_{0,(t_6,t_7,t_8,t_9)} \phi(a_{12}t_{12}) & \text{if $t_1 = 0$.}  \end{array} \right.
\end{equation}
We do not calculate $\lambda_3$ fully here, but do note that
\begin{equation} \label{e:lambda3}
\lambda_3(x_{V_3}(\bt)) =  \delta_{0,(t_5,t_6,t_7)} \lambda_3(x_{V_3}(\bt)).
\end{equation}

(a) It is clear that $\chi(x(\bt)) = 0$, if $t_3 \ne 0$.

(b) Let $x = x_1(c_1)x_2(c_2)x_4(c_4)x_5(e_5)x_6(e_6)x_7(e_7)$ lie in
the family $\cC_{1,2,4,q^6}$ and consider $y = x^{x_3(s_3)}$ for
$s_3 \in \F_q$.  The theory of minimal representatives outlined in
Section \ref{sec:classes} implies that the coefficient of one of
$x_5$, $x_6$ or $x_7$ in $y$ must be nonzero. Therefore,
$
\chi(x) = \sum_{s \in \F_q} \dot \psi_3(x^{x_3(s)}) = 0
$
by \eqref{e:lambda3}.

Similar arguments deal with the other families in the statement.

(c) First we show that if one of $t_1$, $t_2$ or $t_4$ is nonzero, then
\begin{equation} \label{e:red}
\chi\left(\prod_{i \ne 3,5,6,7} x_i(t_i)\right) = \lambda_3\left(\prod_{i \ne 3,5,6,7}
x_i(t_i)\right).
\end{equation}
We just deal with the case $t_1 \ne 0$, as the other cases can be dealt with
similarly.
Consider $y = \prod_{i \ne 3,5,6,7} x_i(t_i)^{x_3(s_3)}$, for $s_3
\in \F_q$. We calculate the coefficient of $x_5$ in $y$ to be $t_1 s_3$.
Therefore, by \eqref{e:lambda3}, the only nonzero
term in the right side of
$$
\chi\left(\prod_{i \ne 3,5,6,7} x_i(t_i)\right) = \sum_{s_3 \in \F_q} \dot
\lambda_3\left(\prod_{i \ne 3,5,6,7} x_i(t_i)^{x_3(s_3)}\right)
$$
is when $s_3 = 0$.

We now give the calculation for $\cC_{1,2,4,q^7}$
\begin{align*}
& \: \chi(x_1(t_1)x_2(t_2)x_4(t_4)x_{12}(t_{12})) \\
=& \: \lambda_3(x_1(t_1)x_2(t_2)x_4(t_4)x_{12}(t_{12})) \\
=& \: \sum_{s_5 \in \F_q} \dot
\lambda_2((x_1(t_1)x_2(t_2)x_4(t_4)x_{12}(t_{12}))^{x_5(s_5)}) \\
=& \:  \sum_{s_5 \in \F_q} \dot \lambda_2(x_1(t_1)x_2(t_2)x_8(t_2s_5)x_4(t_4)x_9(t_4 s_5) x_{12}(t_{12})) \\
=& \: \sum_{s_5 \in \F_q} \dot \lambda_2(x_1(t_1)x_2(t_2)x_4(t_4) x_8(
t_2s_5) x_9(t_4 s_5) x_{11}(-t_2t_4s_5) x_{12}(t_{12})) \\
=& \:  q\phi(a_{12}t_{12}) \sum_{s \in \F_q} \phi(- a_{12}t_1t_2t_4s^2).
\end{align*}
The first equality is given by \eqref{e:red}, and the last uses \eqref{e:lambda2} and
substitution $s = t_1s_5$.

We omit the calculations for the other conjugacy classes in (c), as they are very similar.

(d) This is a straightforward calculation, which we omit.
\end{proof}

For the case $p=2$, we just have to deal with the classes
$\cC_{1,2,4,2q^7}^{p=2}$ and $\cC_{1,2,4,q^7}^{p=2}$, which we do in the following lemma.  The
calculations involved are a minor modification to those for the class
$\cC_{1,2,4}$ for $p>2$, so we omit the proof.

\begin{lemma} \label{l:12p2}
Let $p=2$.  The values of on the classes $\cC_{1,2,4,2q^7}^{p=2}$ and $\cC_{1,2,4,q^7}^{p=2}$
are given as follows.
\begin{itemize}
\item[$\cC_{1,2,4,2q^7}^{p=2}$:]\  $\chi_{12}^{a_{12},b_1,b_2,b_4}(x_1(t_1)x_2(t_2)x_4(t_4)x_{10}(t_{10})x_{12}(t_{12}))
 \\ =q\phi(b_1t_1+b_2t_2+b_4t_4+a_{12}t_{12})\sum_{s \in \F_q} \phi(a_{12}(st_1t_{10} + t_1t_2
t_4s^2))$, for $t_1,t_2,t_4,t_{10} \in \F_q^\times$ and $t_{12} \in \F_q$.
\item[$\cC_{1,2,4,q^7}^{p=2}$:]\  $\chi_{12}^{a_{12},b_1,b_2,b_4}(x_1(t_1)x_2(t_2)x_4(t_4)x_{12}(t_{12}))
= 0$, for $t_1,t_2,t_4 \in \F_q^\times$ and $t_{12} \in \F_q$.
\end{itemize}
\end{lemma}

\end{document}